\documentclass[a4paper]{article}
\usepackage{mathrsfs}
\usepackage{amssymb}
\usepackage{amsmath}
\usepackage{pict2e}
\usepackage{hyperref}
\usepackage{amsmath}
\usepackage{amsfonts}
\usepackage{amsthm}
\usepackage{graphicx}
\usepackage[small]{caption}
\usepackage{subfigure}

\def\0{\emptyset}

\newcommand{\ch}{\text{ch}}

\begin{document}
\newtheorem{claim}{Claim}[section]
\newtheorem{theorem}{Theorem}[section]
\newtheorem{corollary}[theorem]{Corollary}
\newtheorem{definition}[theorem]{Definition}
\newtheorem{conjecture}[theorem]{Conjecture}
\newtheorem{problem}[theorem]{Problem}
\newtheorem{lemma}[theorem]{Lemma}
\newtheorem{proposition}[theorem]{Proposition}
\newcommand{\remark}{\medskip\par\noindent {\bf Remark.~~}}
\newcommand{\pp}{{\it p.}}
\newcommand{\de}{\em}

\title{Planar graphs without cycles of length $4$ or $5$ are $(11:3)$-colorable}
\author{Zden\v{e}k Dvo\v{r}\'{a}k\thanks{Computer Science Institute (CSI) of Charles University,
           Malostransk\'{e} n\'am\v{e}st\'{\i} 25, 118 00 Prague,
	              Czech Republic. E-mail: \protect\href{mailto:rakdver@iuuk.mff.cuni.cz}{\protect\nolinkurl{rakdver@iuuk.mff.cuni.cz}}.}\and Xiaolan Hu\thanks{School of Mathematics and Statistics $\&$ Hubei Key Laboratory of Mathematical Sciences, Central China Normal University,
Wuhan 430079, PR China.}
}
\date{}
\maketitle

\begin{abstract}
A graph $G$ is \emph{$(a:b)$-colorable} if there exists an assignment of
$b$-element subsets of $\{1,\ldots,a\}$ to vertices of $G$ such that
sets assigned to adjacent vertices are disjoint.
We show that every planar graph without cycles of length $4$ or $5$ is
$(11:3)$-colorable, a weakening of recently disproved Steinberg's conjecture.
In particular, each such graph with $n$ vertices has an independent set of size at least $\frac{3}{11}n$.

\end{abstract}

\section{Introduction}

A function that assigns sets to all vertices of a graph is a \emph{set coloring} if the sets
assigned to adjacent vertices are disjoint.
For positive integers $a$ and $b\le a$, an {\em $(a:b)$-coloring} of a graph $G$ is a set coloring
with range $\binom {\{1,\ldots, a\}}{b}$, i.e., a set coloring that to each vertex assigns a $b$-element
subset of $\{1,\ldots, a\}$.
The concept of $(a:b)$-coloring is a generalization of the conventional vertex coloring. In fact,
an $(a:1)$-coloring is exactly an ordinary proper $a$-coloring.

The {\em fractional chromatic number} of $G$, denoted by $\chi_f(G)$, is the infimum of the
fractions $a/b$ such that $G$ admits an $(a:b)$-coloring. Note that $\chi_f(G)\leq \chi(G)$ for any graph $G$,
where $\chi(G)$ is the chromatic number of $G$.
The fractional coloring was first introduced in 1973 \cite{planfr5} to seek for a proof of the Four Color Problem.
Since then, it has been the focus of many intensive research efforts, see \cite{ScheinermanUllman2011}.
In particular, fractional coloring of planar graphs without cycles of certain lengths is widely studied.
Pirnazar and Ullman~\cite{PU02} showed that the fractional chromatic number of a planar
graph with girth at least $8k-4$ is at most $2+\frac{1}{k}$. Dvo\v{r}\'{a}k {\em et al.}~\cite{DSV08}
showed that every planar graph of odd-girth at least 9 is $(5:2)$-colorable. Recently, Dvo\v{r}\'{a}k {\em et al.} \cite{frpltr} showed
that every planar triangle-free graph on $n$ vertices is $(9n:3n+1)$-colorable, and thus it has fractional chromatic number
at most $3-\frac{3}{3n+1}$.

Well-known Steinberg's Conjecture asserts that every planar graph without cycles of length 4 or 5 is 3-colorable.
Recently, Steinberg's conjecture was disproved \cite{CohenAddad2016}. This conjecture, though disproved, had motivated
a lot of research, see \cite{borsurvey}. Since $\chi_f(G)\leq \chi(G)$ for any graph $G$, it is natural to ask whether there exists
a constant $c<4$ such that $x_f(G)\leq c$ for all planar graphs without cycles of length 4 or 5.
In this paper, we confirm this is the case for $c=\frac{11}{3}$. In fact, we prove the following stronger theorem.

\begin{theorem}
\label{MT}
Every planar graph without cycles of length 4 or 5 is $(11:3)$-colorable, and thus its fractional chromatic number
is at most $\frac{11}{3}$.
\end{theorem}

The \emph{independence number} $\alpha(G)$ of a graph $G$ is the size of a largest independent set in $G$.
The \emph{independence ratio} of $G$ is the quantity $\frac{\alpha(G)}{|V(G)|}$.
The famous Four Color Theorem \cite{AppHak1} implies that every planar graph has independence ratio at least $\frac{1}{4}$.
In 1976,  Albertson \cite{Albertson76} proved a weaker result that every planar graph has independence ratio at least $\frac{2}{9}$
without using the Four Color Theorem. In 2016, Cranston and Rabern \cite{CR16} improved this constant to $\frac{3}{13}$.
If $G$ is a triangle-free planar graph, a classical theorem of Gr\H{o}tzsch \cite{grotzsch1959} says that $G$ is 3-colorable,
and thus $G$ has independence ratio at least $\frac{1}{3}$.  This bound can be slightly improved---Steinberg and Tovey~\cite{SteinbergTovey1993}
proved that the independence ratio is at least $\frac{1}{3}+\frac{1}{3|V(G)|}$, and gave an infinite
family of planar triangle-free graphs for that this bound is tight.
Steinberg's Conjecture would imply that every planar graph without cycles of length 4 or 5 has independence ratio at least $\frac{1}{3}$,
and it is not known whether this weaker statement holds or not.
Since $\alpha(G)\geq \frac{|V(G)|}{\chi_f(G)}$ for any graph $G$,
we have the following corollary by Theorem~\ref{MT}.

\begin{corollary}\label{ratio}
Every planar graph without cycles of length 4 or 5 has independence ratio at least $\frac{3}{11}$.
\end{corollary}

It is not clear whether the constant $\frac{11}{3}$ from Theorem~\ref{MT} is the best possible, and we suspect this is not the case.
Hence, the following question is of interest.
\begin{problem}
What is the infimum of fractional chromatic numbers of planar graphs without cycles of length 4 or 5?
\end{problem}
Let us remark that the counterexample to Steinberg's conjecture constructed in \cite{CohenAddad2016} is $(6:2)$-colorable,
and thus we cannot even exclude the possibility that the answer is $3$.

The proof of Theorem~\ref{MT} naturally proceeds in list coloring setting.
A \emph{list assignment} for a graph $G$ is a function $L$ that to each vertex $v$ of $G$ assigns a set $L(v)$ of colors.
A set coloring $\varphi$ of $G$ is an \emph{$L$-set coloring} if $\varphi(v)\subseteq L(v)$ for all $v\in V(G)$.
For a positive integer $b$, we say that $\varphi$ is an \emph{$(L:b)$-coloring} of $G$ if $\varphi$ is an $L$-set coloring
and $|\varphi(v)|=b$ for all $v\in V(G)$.  If such an $(L:b)$-coloring exists, we say that $G$ is \emph{$(L:b)$-colorable}.
For an integer $a\ge b$, we say that $G$ is \emph{$(a:b)$-choosable} if $G$ is $(L:b)$-colorable from any
assignment $L$ of lists of size $a$.  We actually prove the following strengthening of Theorem~\ref{MT}.

\begin{theorem}\label{MTl}
Every planar graph without cycles of length 4 or 5 is $(11:3)$-choosable.
\end{theorem}

\section{Colorability of small graphs}

Let us start with some technical results on list-colorability of small graphs, especially paths and cycles.
In the proofs, it is convenient to work with a non-uniform version of set coloring.  Let $f:V(G)\to\mathbf{Z}_0^+$ be an arbitrary
function.  An \emph{$(L:f)$-coloring} of a graph $G$ is an $L$-set coloring $\varphi$ such that $|\varphi(v)|=f(v)$ for all $v\in V(G)$.
If such an $(L:f)$-coloring exists, we say that $G$ is \emph{$(L:f)$-colorable}.
We repeatedly use the following simple observation.
\begin{lemma}\label{lemma-redulist}
Let $L$ be an assignment of lists to vertices of a graph $G$, let $f$ assign non-negative integers to vertices of $G$,
and let $\psi$ be an $L$-set coloring of $G$ such that $|\psi(v)|\le f(v)$ for all $v\in V(G)$.  Let $L'$ be the list assignment defined by
$$L'(v)=L(v)\setminus \Big(\psi(v)\cup\bigcup_{u\in N_G(v)} \psi(u)\Big)$$
for all $v\in V(G)$, and let $f'(v)=f(v)-|\psi(v)|$ for all $v\in V(G)$.
If $G$ is $(L',f')$-colorable, then $G$ has an $(L:f)$-coloring $\varphi$ such that
$\psi(v)\subseteq \varphi(v)$ for all $v\in V(G)$.
\end{lemma}
\begin{proof}
If $\varphi'$ is an $(L',f')$-coloring of $G$, it suffices to set
$\varphi(v)=\psi(v)\cup\varphi'(v)$ for all $v\in V(G)$.
\end{proof}

We also use the following observation.
\begin{lemma}\label{lemma-greedy}
Let $L$ be an assignment of lists to vertices of a graph $G$, let $f$ assign non-negative integers to vertices of $G$,
and let $v_1$, \ldots, $v_n$ be an ordering of vertices of $G$.
If
\begin{equation}\label{eq:assgreedy}
|L(v_i)|\ge f(v_i)+\sum_{v_jv_i\in E(G), j<i} f(v_j)
\end{equation}
holds for $1\le i\le n$, then $G$ has an $(L:f)$-coloring.
\end{lemma}
\begin{proof}
We prove the claim by induction on $n$. The basic case $n=0$ is trivial.  If $n\ge 1$, then $|L(v_1)|\ge f(v_1)$ by the assumptions,
and thus there exists a subset $A$ of $L(v_1)$ of size $f(v_1)$.  Let $L'(v_i)=L(v_i)\setminus A$ for all $i$
such that $v_1v_i\in E(G)$, and $L'(v_i)=L(v_i)$ for all $i\ge 2$ such that $v_1v_i\not\in E(G)$.
Since $|L'(v_i)|\ge |L(v_i)|-f(v_1)$ in the former case and $|L'(v_i)|=|L(v_i)|$ in the latter case,
it is easy to verify that the assumption (\ref{eq:assgreedy}) holds for $G-v_1$ with the vertex ordering $v_2$, \ldots, $v_n$
and the list assignment $L'$.
Hence, by the induction hypothesis, $G-v_1$ has an $(L':f)$-coloring. Assigning $A$ to $v_1$ turns this coloring into an $(L:f)$-coloring of $G$.
\end{proof}
When Lemma~\ref{lemma-greedy} applies, we say that we \emph{color vertices of $G$ greedily in order $v_1$, \ldots, $v_n$}.

Finally, let us make another simple observation, which we will often (implicitly) apply.
Let $G$ be a graph, let $G_0$ be a subgraph of $G$, and let $f,g:V(G)\to\mathbf{Z}_0^+$ be functions such that
$f(v)\le g(v)$ for all $v\in V(G)$.
Let us consider the situation that we need to prove that a graph is $(L:f)$-colorable
for every list assignment $L$ such that $|L(v)|\ge g(v)$ for all $v\in V(G)$, under assumption that $G_0$ is $(L:f)$-colorable.
Then it suffices to prove this for all list assignments $L(v)$ such that $|L(v)|=g(v)$ for all $v\in V(G)$: if $|L(v)|>g(v)$,
then we can without loss of generality throw away any color in $L(v)$, not used in the $(L:f)$-coloring of $G_0$ when $v\in V(G_0)$.

\begin{lemma}\label{3-3-3}
Let $L$ be a list assignment for a path $P=v_1v_2v_3$.
If $|L(v_1)|=|L(v_3)|=5$ and $|L(v_2)|=8$, then $P$ is $(L:3)$-colorable.
Moreover, for any colors $\alpha_1,\alpha_2\in L(v_1)$ and $\beta\in L(v_3)$,
there exists an $(L:3)$-coloring $\varphi$ of $P$ such that $\alpha_1,\alpha_2\in \varphi(v_1)$
and $\beta\in \varphi(v_3)$.
\end{lemma}
\begin{proof}
Consider arbitrary colors $\alpha_1,\alpha_2\in L(v_1)$ and $\beta\in L(v_3)$.
Let $f'(v_1)=1$, $f'(v_2)=3$, and $f'(v_3)=2$.
By Lemma~\ref{lemma-redulist}, it suffices to prove that $P$ has
an $(L':f')$-coloring for any list assignment $L'$ such that $|L'(v_1)|=3$, $|L'(v_2)|=5$, and $|L'(v_3)|=4$.

Choose colors $\gamma_1\in L'(v_2)\setminus L'(v_3)$ and $\gamma_2\in L'(v_2)\setminus (\{\gamma_1\}\cup L'(v_1))$.
Choose $\varphi'(v_2)$ as any $3$-element subset of $L'(v_2)$ containing $\gamma_1$ and $\gamma_2$.
Then $|L'(v_1)\setminus \varphi'(v_2)|\ge 1$ and $|L'(v_3)\setminus \varphi'(v_2)|\ge 2$,
and thus we can choose $\varphi'(v_1)$ as a $1$-element subset of $L'(v_1)\setminus \varphi'(v_2)$
and $\varphi'(v_3)$ as a $2$-element subset of $L'(v_3)\setminus \varphi'(v_2)$.
Clearly, $\varphi'$ is an $(L':f')$-coloring of $P$.
\end{proof}

\begin{lemma}\label{3-3-4-3}
Let $L$ be a list assignment for a path $P=v_1v_2v_3v_4$
such that $|L(v_1)|=|L(v_3)|=|L(v_4)|=5$, $|L(v_2)|=8$ and the subpath $v_3v_4$ of $P$
has an $(L:3)$-coloring.  Then $P$ is $(L:3)$-colorable.
Moreover, for any colors $\alpha\in L(v_1)$ and $\beta\in L(v_4)$,
the path $P$ has an $(L:3)$-coloring $\varphi$ such that $\alpha\in \varphi(v_1)$ and $\beta\in \varphi(v_4)$.
\end{lemma}

\begin{proof}
Since the path $v_3v_4$ is $(L:3)$-colorable, we have $L(v_3)\neq L(v_4)$.
Consider arbitrary colors $\alpha\in L(v_1)$ and $\beta,\beta'\in L(v_4)$
such that at most one of the colors $\beta$ and $\beta'$ belongs to $L(v_3)$.
Let $f'(v_1)=2$, $f'(v_2)=f'(v_3)=3$ and $f'(v_4)=1$.  By Lemma~\ref{lemma-redulist}, it suffices to prove that $P$ has
an $(L':f')$-coloring for any list assignment $L'$ such that $|L'(v_1)|=4$, $|L'(v_2)|=7$, $|L'(v_3)|=4$, and $|L'(v_4)|=3$.

Let $\gamma_3$ be any color in $L'(v_3)\setminus L'(v_4)$ and let $\gamma_2$ be any color in $L'(v_2)\setminus (\{\gamma_3\}\cup L'(v_1))$.
Let $f''(v_1)=f''(v_2)=f''(v_3)=2$ and $f''(v_4)=1$.
By Lemma~\ref{lemma-redulist}, it suffices to prove that $P$ has
an $(L'':f'')$-coloring for any list assignment $L''$ such that $|L''(v_1)|=4$, $|L''(v_2)|=5$, $|L''(v_3)|=2$, and $|L''(v_4)|=3$.
This is the case by coloring the vertices of $P$ greedily in order $v_3$, $v_4$, $v_2$, and $v_1$.
\end{proof}

\begin{lemma} \label{3-3-4-4-3}
Let $L$ be a list assignment for a path $P=v_1\ldots v_5$
such that $|L(v_1)|=|L(v_3)|=|L(v_4)|=|L(v_5)|=5$, $|L(v_2)|=8$,
and the subpath $v_3v_4v_5$ has an $(L:3)$-coloring.
Then $P$ is $(L:3)$-colorable.  Moreover, for any colors $\alpha\in L(v_1)$ and $\beta\in L(v_5)$ such that $\{\beta\}\neq L(v_4)\setminus L(v_3)$,
the path $P$ has an $(L:3)$-coloring $\varphi$ such that $\alpha\in \varphi(v_1)$ and $\beta\in \varphi(v_5)$.
\end{lemma}
\begin{proof}
Since the path $v_3v_4v_5$ is $(L:3)$-colorable, we have $L(v_3)\neq L(v_4)\neq L(v_5)$.
Consider arbitrary colors $\alpha\in L(v_1)$, $\varepsilon\in L(v_3)\setminus L(v_4)$, and $\beta\in L(v_5)$
such that $\{\beta\}\neq L(v_4)\setminus L(v_3)$.
There exists a color $\gamma\in L(v_4)\setminus L(v_3)$ such that $\gamma\neq \beta$.
If $\beta\not\in L(v_4)$, then choose $\beta'\in L(v_5)\setminus\{\beta,\gamma\}$ arbitrarily;
otherwise, choose $\beta'\in L(v_5)\setminus L(v_4)$ arbitrarily.  In either case,
assigning sets $\{\alpha\}$, $\emptyset$, $\{\varepsilon\}$, $\{\gamma\}$, $\{\beta,\beta'\}$ to vertices of $P$ in order
gives an $L$-set coloring.  Let $f'(v_1)=f'(v_3)=f'(v_4)=2$, $f'(v_2)=3$, and $f'(v_5)=1$.
By Lemma~\ref{lemma-redulist}, it suffices to prove that $P$ has
an $(L':f')$-coloring for any list assignment $L'$ such that $|L'(v_1)|=4$, $|L'(v_2)|=6$, $|L'(v_3)|=4$, $|L'(v_4)|=3$, and $|L'(v_5)|=2$.

Choose $\kappa_3\in L'(v_3)\setminus L'(v_4)$, $\kappa_4\in L'(v_4)\setminus L'(v_5)$, and $\kappa_2\in L'(v_2)\setminus (\{\kappa_3\}\cup L'(v_1))$.
Let $f''(v_1)=f''(v_2)=2$ and $f''(v_3)=f''(v_4)=f''(v_5)=1$.  By Lemma~\ref{lemma-redulist}, it suffices to prove that $P$ has
an $(L'':f'')$-coloring for any list assignment $L''$ such that $|L''(v_1)|=|L''(v_2)|=4$, $|L''(v_3)|=1$, and $|L''(v_4)|=|L''(v_5)|=2$.
This is the case by coloring the vertices of $P$ greedily in order $v_3$, $v_4$, $v_5$, $v_2$, and $v_1$.
\end{proof}

\begin{lemma}
\label{3-3-4-4-4-3}
Let $L$ be a list assignment for a path $P=v_1\ldots v_6$
such that $|L(v_1)|=|L(v_3)|=\ldots=|L(v_6)|=5$, $|L(v_2)|=8$,
and the subpath $v_3\ldots v_6$ has an $(L:3)$-coloring.
Then $P$ is $(L:3)$-colorable.
Moreover, for any color $\alpha\in L(v_1)$, the path $P$ has an $(L:3)$-coloring
$\varphi$ such that $\alpha\in \varphi(v_1)$.
\end{lemma}
\begin{proof}
Since $v_3v_4v_5v_6$ has an $(L:3)$-coloring $\psi$, we have $L(v_3)\neq L(v_4)\neq L(v_5)\neq L(v_6)$.
Furthermore, if $|L(v_4)\setminus L(v_3)|=1$, then $\psi(v_4)$ contains the unique color $\gamma\in L(v_4)\setminus L(v_3)$,
and thus $L(v_5)\setminus \{\gamma\}\not\subseteq L(v_6)$; in this case, let $\beta$ be an arbitrary color in
$L(v_5)\setminus (\{\gamma\}\cup L(v_6))$.  Otherwise, let $\beta$ be an arbitrary color in $L(v_5)\setminus L(v_6)$.

In either case, we have $\{\beta\}\neq L(v_4)\setminus L(v_3)$; hence, considering any $\alpha\in L(v_1)$,
$P-v_6$ has an $(L:3)$-coloring $\varphi$ such that $\alpha\in \varphi(v_1)$ and $\beta\in \varphi(v_5)$
by Lemma~\ref{3-3-4-4-3}.  Since $\beta\not\in L(v_6)$, we have $|L(v_6)\setminus\varphi(v_5)|\ge 3$,
and thus $\varphi$ can be extended to an $(L:3)$-coloring of $P$ by choosing $\varphi(v_6)$ as an
arbitrary $3$-element subset of $L(v_6)\setminus\varphi(v_5)$.
\end{proof}

\begin{lemma}\label{3-4-3---3}
Let $L$ be a list assignment for a path $P=v_1\ldots v_k$ with $5\le k\le 7$
such that $|L(v_1)|=|L(v_2)|=|L(v_4)|=\ldots=|L(v_k)|=5$, $|L(v_3)|=8$,
and $P-v_3$ has an $(L:3)$-coloring.
Then $P$ is $(L:3)$-colorable.
\end{lemma}
\begin{proof}
Since $v_1v_2$ has an $(L:3)$-coloring, we have $L(v_1)\neq L(v_2)$, and thus there exists
a color $\alpha\in L(v_2)\setminus L(v_1)$.  By Lemmas~\ref{3-3-4-3}, \ref{3-3-4-4-3}, and \ref{3-3-4-4-4-3},
there exists an $(L:3)$-coloring $\varphi$ of $P-v_1$ such that $\alpha\in L(v_2)$.
Since $\alpha\not\in L(v_1)$, we have $|L(v_1)\setminus\varphi(v_2)|\ge 3$,
and thus $\varphi$ can be extended to an $(L:3)$-coloring of $P$ by choosing $\varphi(v_1)$ as an
arbitrary $3$-element subset of $L(v_1)\setminus\varphi(v_2)$.
\end{proof}

\begin{lemma}\label{triangle}
Let $L$ be a list assignment for a triangle $C=v_1v_2v_3$.
Then $C$ is $(L:3)$-colorable if and only if
$|L(v_i)|\ge 3$ for $1\le i\le 3$,
$|L(v_i)\cup L(v_j)|\ge 6$ for $1\le i < j\le 3$,
and $|L(v_1)\cup L(v_2)\cup L(v_3)|\ge 9$.
\end{lemma}
\begin{proof}
If $\varphi$ is an $(L:3)$-coloring of $C$ and $S$ is a subset of $V(C)$,
then $\varphi$ assigns pairwise disjoint sets to vertices of $S$, and
thus $\big|\bigcup_{v\in S} L(v)\big|\ge \big|\bigcup_{v\in S} \varphi(v)\big|=3|S|$,
proving that the conditions from the statement of the lemma are necessary.

Consider an auxiliary bipartite graph $H$ with one part $U$ consisting of
$L(v_1)\cup L(v_2)\cup L(v_3)$ and the other part $V$ consisting of
vertices $v_{i,k}$ for $1\le i,k\le 3$, with $c\in U$ adjacent to $v_{i,k}$
if and only if $c\in L(v_i)$.  Using Hall's theorem, the assumptions of the lemma
imply that $H$ has a matching saturating the vertices of $V$.
Letting $\varphi(v_i)$ consist of the colors joined to $v_{i,1}$, $v_{i,2}$, and $v_{i,3}$
in this matching for $1\le i\le 3$ gives an $(L:3)$-coloring of $C$.
\end{proof}

\begin{lemma}\label{lollipop}
Let $L$ be a list assignment for the graph $H$ consisting of a path $v_1v_2v_3v_4$ and an edge $v_1v_3$,
such that $|L(v_1)|=|L(v_4)|=5$, $|L(v_2)|=|L(v_3)|=8$,
and the triangle $v_1v_2v_3$ has an $(L:3)$-coloring.
Then $H$ is $(L:3)$-colorable.
\end{lemma}
\begin{proof}
Since $v_1v_2v_3$ is $(L:3)$-colorable, we have $|L(v_1)\cup L(v_2)\cup L(v_3)|\ge 9$ by Lemma~\ref{triangle},
and thus there exists a color $\alpha\in (L(v_1)\cup L(v_3))\setminus L(v_2)$.
Let $\beta$ be a color in $L(v_3)\setminus L(v_1)$.
Let $\varphi(v_4)$ be any $3$-element subset of $L(v_4)\setminus\{\alpha,\beta\}$.
Let $L'(v_3)=L(v_3)\setminus \varphi(v_4)$, $L'(v_1)=L(v_1)$ and $L'(v_2)=L(v_2)$.
Note that $\beta\in L'(v_3)\setminus L'(v_1)$, and thus $|L'(v_1)\cup L'(v_3)|\ge 6$.
Furthermore, $\alpha\in (L'(v_1)\cup L'(v_3))\setminus L'(v_2)$, and thus
$|L'(v_1)\cup L'(v_2)\cup L'(v_3)|\ge |L'(v_2)|+1=9$.  By Lemma~\ref{triangle},
$v_1v_2v_3$ has an $(L':3)$-coloring, and this coloring extends $\varphi$ to an $(L:3)$-coloring of $H$.
\end{proof}

\begin{lemma}\label{l6cycle}
Let $L$ be a list assignment for a $6$-cycle $C=v_1\ldots v_6$, such that
$|L(v_i)|\ge 5$ for $1\le i\le 6$.  Suppose that there exists $S\subseteq V(C)$
such that $|S|=2$, $|L(u)|=8$ for all $u\in S$, and $C-S$ is $(L:3)$-colorable.
Then $C$ is $(L:3)$-colorable.
\end{lemma}
\begin{proof}
Without loss of generality, we can assume $|L(v)|=5$ for $v\in V(C)\setminus S$
and $S=\{v_1,v_t\}$ for some $t\in\{2,3,4\}$.  Let us discuss the possible values of
$t$ separately.
\begin{itemize}
\item Suppose first that $t=2$.  Since $C-S$ is $(L:3)$-colorable, we have
$L(v_3)\neq L(v_4)\neq L(v_5)\neq L(v_6)$, and furthermore, if $|L(v_4)\setminus L(v_3)|=1$
and $|L(v_5)\setminus L(v_6)|=1$, then $L(v_4)\setminus L(v_3)\neq L(v_5)\setminus L(v_6)$.
Select $\beta\in L(v_4)\setminus L(v_3)$ such that $|L(v_5)\setminus (\{\beta\}\cup L(v_6))|\ge 1$,
and let $\gamma\in L(v_5)\setminus (\{\beta\}\cup L(v_6))$ be arbitrary.
Then, select $\beta'\in L(v_4)\setminus\{\beta,\gamma\}$ so that at most one of $\beta$ and $\beta'$ belongs to $L(v_5)$,
and $\gamma'\in L(v_5)\setminus \{\beta,\beta',\gamma\}$ so that at most one of $\gamma$ and $\gamma'$ belongs to $L(v_4)$.
Furthermore, arbitrarily select $\alpha\in L(v_3)\setminus L(v_4)$ and $\varepsilon\in L(v_6)\setminus L(v_5)$.
Note that assignment of sets $\emptyset$, $\emptyset$,
$\{\alpha\}$, $\{\beta,\beta'\}$, $\{\gamma,\gamma'\}$, $\{\varepsilon\}$ to vertices of $C$ in order
is a set coloring.
Let $f'(v_1)=f'(v_2)=3$, $f'(v_3)=f'(v_6)=2$, and $f'(v_4)=f'(v_5)=1$.
By Lemma~\ref{lemma-redulist}, it suffices to prove that $C$ has an $(L':f')$-coloring for any list assignment $L'$ such that
$|L'(v_1)|=|L'(v_2)|=7$, $|L'(v_3)|=|L'(v_6)|=3$, and $|L'(v_4)|=|L'(v_5)|=2$.

Choose $\alpha'\in L'(v_3)\setminus L'(v_4)$ and $\varepsilon'\in L'(v_6)\setminus L'(v_5)$.
Let $f''(v_1)=f''(v_2)=3$ and $f''(v_3)=\ldots=f''(v_6)=1$.  Applying Lemma~\ref{lemma-redulist} again,
it suffices to prove that $C$ has an $(L'':f'')$-coloring for any list assignment $L''$ such that
$|L''(v_1)|=|L''(v_2)|=6$ and $|L''(v_3)|=\ldots=|L''(v_6)|=2$.  If $L''(v_1)\neq L''(v_2)$,
then let $\kappa$ be a color in $L''(v_1)\setminus L''(v_2)$, and let $\varphi$ be an $(L'':f'')$-coloring
of the path $v_6v_5v_4v_3$ such that $\varphi(v_6)\neq \{\kappa\}$, obtained greedily.
If $L''(v_1)=L''(v_2)$, then let $\varphi$ be an $(L'':f'')$-coloring
of the path $v_6v_5v_4v_3$ such that $\varphi(v_3)\neq\varphi(v_6)$, which exists, since a $4$-cycle
is $(2:1)$-choosable~\cite{erdosrubintaylor1979}.
In either case, the choice of $\varphi$ ensures that if $L''(v_1)\setminus\varphi(v_6)$ and $L''(v_2)\setminus \varphi(v_3)$
both have size $5$, then they are different.  Hence, we can choose $\varphi(v_1)$ and $\varphi(v_2)$ as disjoint
$3$-element subsets of $L''(v_1)\setminus\varphi(v_6)$ and $L''(v_2)\setminus \varphi(v_3)$, respectively.
This gives an $(L'':f'')$-coloring of $C$, as required.

\item  Next, suppose that $t=3$.  If $L(v_2)\not\subseteq L(v_1)$, then choose $\alpha\in L(v_2)\setminus L(v_1)$
and $\beta\in L(v_6)$ arbitrarily so that $L(v_5)\setminus L(v_4)\neq\{\beta\}$.  If $L(v_2)\subseteq L(v_1)$,
then note that $|L(v_6)\setminus (L(v_1)\setminus L(v_2))|\ge 2$, and thus we can choose $\beta\in L(v_6)\setminus (L(v_1)\setminus L(v_2))$
so that $L(v_5)\setminus L(v_4)\neq\{\beta\}$.  In this case, if $\beta\in L(v_2)$ then let $\alpha=\beta$, otherwise
choose $\alpha\in L(v_2)$ arbitrarily.  By Lemma~\ref{3-3-4-4-3}, the path $v_2v_3\ldots v_6$ has an $(L:3)$-coloring $\varphi$
such that $\alpha\in \varphi(v_2)$ and $\beta\in \varphi(v_6)$.  By the choice of $\alpha$ and $\beta$ we have
$|L(v_1)\setminus (\varphi(v_2)\cup \varphi(v_6))|\ge 3$, and thus we can extend $\varphi$ to an $(L:3)$-coloring of $C$
by choosing $\varphi(v_1)$ as a $3$-element subset of $L(v_1)\setminus (\varphi(v_2)\cup \varphi(v_6))$.

\item Finally, suppose that $t=4$.  Since $C-S$ is $(L:3)$-colorable, we have
$L(v_2)\neq L(v_3)$ and $L(v_5)\neq L(v_6)$.  Hence, there exist $\alpha\in L(v_2)\setminus L(v_3)$, $\beta\in L(v_3)\setminus L(v_2)$,
$\gamma\in L(v_5)\setminus L(v_6)$, and $\varepsilon\in L(v_6)\setminus L(v_5)$.
Let $f'(v_1)=f'(v_4)=3$ and $f'(v_2)=f'(v_3)=f'(v_5)=f'(v_6)=2$.  By Lemma~\ref{lemma-redulist}, it suffices to prove that $C$ has an $(L':f')$-coloring for any list assignment $L'$ such that
$|L'(v_1)|=|L'(v_4)|=6$ and  $|L'(v_2)|=|L'(v_3)|=|L'(v_5)|=|L'(v_6)|=4$.

Suppose first that there exists a color $\alpha'\in L'(v_2)\setminus L'(v_1)$.  Since $|L'(v_5)|+|L'(v_3)\setminus\{\alpha'\}|>|L'(v_4)|$,
there exist colors $\beta'\in L'(v_3)\setminus\{\alpha'\}$ and $\gamma'\in L'(v_5)$ such that either $\beta'=\gamma'$ or at most one of $\beta'$ and $\gamma'$ belongs to $L'(v_4)$.
Let $f''(v_1)=f''(v_4)=3$, $f''(v_2)=f''(v_3)=f''(v_5)=1$, and $f''(v_6)=2$.
By Lemma~\ref{lemma-redulist}, it suffices to prove that $C$ has an $(L'':f'')$-coloring for any list assignment $L''$ such that
$|L''(v_1)|=6$, $|L''(v_4)|=5$, $|L''(v_2)|=|L''(v_3)|=2$, and $|L''(v_5)|=|L''(v_6)|=3$.
This is the case by coloring the vertices of $C$ greedily in order $v_2$, $v_3$, $v_5$, $v_6$, $v_1$, and $v_4$.

Hence, we can assume that $L'(v_2)\subset L'(v_1)$, and by symmetry also $L'(v_6)\subset L'(v_1)$ and $L'(v_3),L'(v_5)\subset L'(v_4)$.
Since $|L'(v_2)|+|L'(v_6)|-|L'(v_1)|=2$, it follows that $|L'(v_2)\cap L'(v_6)|\ge 2$, and symmetrically $|L'(v_3)\cap L'(v_5)|\ge 2$.
Hence, there exist $\alpha'\in L'(v_2)\cap L'(v_6)$ and $\beta'\in L'(v_3)\cap L'(v_5)$ such that $\alpha'\neq\beta'$.
Let $f''(v_1)=f''(v_4)=3$ and $f''(v_2)=f''(v_3)=f''(v_5)=f''(v_6)=1$.
By Lemma~\ref{lemma-redulist}, it suffices to prove that $C$ has an $(L'':f'')$-coloring for any list assignment $L''$ such that
$|L''(v_1)|=|L''(v_4)|=5$ and $|L''(v_2)=|L''(v_3)|=|L''(v_5)|=|L''(v_6)|=2$.
This is again the case by coloring the vertices of $C$ greedily in order $v_2$, $v_3$, $v_5$, $v_6$, $v_1$, and $v_4$.
\end{itemize}
\end{proof}

\begin{lemma}\label{claw}
Let $L$ be a list assignment for the graph $H$ consisting of a vertex $v$ with three neighbors $v_1$, $v_2$, and $v_3$,
such that $|L(v_i)|=5$ for $1\le i\le 3$ and $|L(v)|=8$.
Then $H$ is $(L:3)$-colorable.
\end{lemma}
\begin{proof}
For $1\le i\le 3$, there exists $\alpha_i\in L(v)\setminus L(v_i)$.
Fix $\varphi(v)$ as a $3$-element subset of $L(v)$ containing $\alpha_1$, $\alpha_2$, and $\alpha_3$.
Then $|L(v_i)\setminus\varphi(v)|\ge 3$ for $1\le i\le 3$, and thus $\varphi$ extends to an $(L:3)$-coloring of $H$.
\end{proof}

\begin{lemma}\label{claw5}
Let $L$ be a list assignment for the graph $H$ consisting of a vertex $v$ with four neighbors $v_1$, \ldots, $v_4$,
and possibly the edge $v_3v_4$, such that $|L(v_1)|=|L(v_2)|=5$, $|L(v)|=8$ and either
\begin{itemize}
\item $v_3v_4\not\in E(H)$ and $|L(v_3)|=|L(v_4)|=5$, or
\item $v_3v_4\in E(H)$, $|L(v_3)|=|L(v_4)|=8$, and the triangle $vv_3v_4$ is $(L:3)$-colorable.
\end{itemize}
Then $H$ is $(L:3)$-colorable.
\end{lemma}
\begin{proof}
If $v_3v_4\not\in E(H)$, then let $A_i$ be a $3$-element subset of $L(v)\setminus L(v_i)$ for $1\le i\le 4$.
Since $\sum_{i=1}^4 |A_i|>|L(v)|$, there exists a color in $L(v)$ belonging to at least two of the sets $A_1$, \ldots, $A_4$.
Hence, there exists a $3$-element set $\varphi(v)\subset L(v)$ such that $\varphi(v)\cap A_i\neq \emptyset$ for $1\le i\le 4$.
Then $|L(v_i)\setminus \varphi(v)|\ge 3$, and thus $\varphi$ extends to an $(L:3)$-coloring of $H$.

If $v_3v_4\in E(H)$, then since $vv_3v_4$ is $(L:3)$-colorable, there exists a color $\alpha\in L(v)$
such that $|\{\alpha\}\cup L(v_3)\cup L(v_4)|\ge 9$.  Since $|L(v_1)\setminus \{\alpha\}|+|L(v_2)\setminus \{\alpha\}|>|L(v)\setminus\{\alpha\}|$,
there exist colors $\beta_1\in L(v_1)\setminus\{\alpha\}$ and $\beta_2\in L(v_2)\setminus\{\alpha\}$ such that either $\beta_1=\beta_2$ or
at most one of $\beta_1$ and $\beta_2$ belongs to $L(v)$.  For $i\in\{1,2\}$, let $\varphi(v_i)$ be any $3$-element subset of $L(v_i)\setminus\{\alpha\}$
containing $\beta_i$.  Then $L(v)\setminus (\varphi(v_1)\cup\varphi(v_2))$ has size at least $3$ and contains $\alpha$, and thus
$\varphi$ extends to an $(L:3)$-coloring of $H$ by Lemma~\ref{triangle}.
\end{proof}

\begin{lemma}\label{claw53}
Let $L$ be a list assignment for the graph $H$ consisting of a vertex $v$ with three neighbors $v_1$, $v_2$, and $v_3$,
a vertex $u_1$ adjacent to $v_1$, and possibly one edge between the vertices $v_1$, $v_2$, and $v_3$,
such that $|L(u_1)|=|L(v)|=5$ and $|L(v_i)|=2+3\deg_H(v_i)$ for $1\le i\le 3$.
If $H-v_1$ is $(L:3)$-colorable, then $H$ is $(L:3)$-colorable.
\end{lemma}
\begin{proof}
If $v_1v_2\in E(H)$, then let $\varphi(v_2)$ be a $3$-element subset of $L(v_2)\setminus L(v)$.
Let $L'(v)=L(v)$, $L'(v_3)=L(v_3)$, $L'(u_1)=L(u_1)$, and $L'(v_1)=L(v_1)\setminus\varphi(v_2)$.
Since $H-v_1$ is $L$-colorable, $H-\{v_1,v_2\}$ is $L'$-colorable, and thus $H-v_2$ is
$L'$-colorable by Lemma~\ref{3-3-4-3}.  Hence, $\varphi$ extends to an $(L:3)$-coloring of $H$.

Hence, assume that $v_1v_2\not\in E(H)$, and by symmetry, $v_1v_3\not\in E(H)$.
If $v_2v_3\in E(H)$, then since $vv_2v_3$ is $(L:3)$-colorable, there exists $\alpha\in L(v)$
such that $|\{\alpha\}\cup L(v_2)\cup L(v_3)|\ge 9$.  Choose distinct $\beta,\beta'\in L(v_1)\setminus\{\alpha\}$
such that $\beta\not\in L(u_1)$ and $\beta'\not\in L(v)$.
Let $\varphi(v_1)$ be an arbitrary $3$-element subset of $L(v_1)\setminus\{\alpha\}$ containing
$\beta$ and $\beta'$.  Then $\varphi$ extends to an $(L:3)$-coloring of $H$ (using Lemma~\ref{triangle}), since $|L(u_1)\setminus\varphi(v_1)|\ge 3$,
$|L(v)\setminus\varphi(v_1)|\ge 3$, and $\alpha\in L(v)\setminus\varphi(v_1)$.

Finally, suppose that $\{v_1, v_2,v_3\}$ is an independent set.  Since the path $v_2vv_3$ is $(L:3)$-colorable,
we have $L(v_2)\neq L(v)\neq L(v_3)$.  Choose colors $\beta\in L(v)\setminus L(v_2)$ and $\beta'\in L(v)\setminus L(v_3)$
arbitrarily.  By Lemma~\ref{3-3-3}, there exists an $(L:3)$-coloring $\varphi$ of
$vv_1u_1$ such that $\beta,\beta'\in \varphi(v)$.  Then $|L(v_i)\setminus\varphi(v)|\ge 3$ for $i\in \{2,3\}$,
and thus $\varphi$ extends to an $(L:3)$-coloring of $H$.
\end{proof}

\begin{lemma}\label{claw543}
Let $L$ be a list assignment for the graph $H$ consisting of a path $u_1v_1vv_2u_2$, a vertex $v_3$ adjacent to $v$,
and possibly the edge $v_1v_3$,
such that $|L(u_1)|=|L(v)|=|L(u_2)|=5$, $|L(v_2)|=8$, $|L(v_3)|=2+3\deg_H(v_3)$, and $|L(v_1)|=3\deg_H(v_1)-1$.
If $H-v_2$ is $(L:3)$-colorable, then $H$ is $(L:3)$-colorable.
\end{lemma}
\begin{proof}
Let us first consider the case that $v_1v_3\in E(H)$.  Since the triangle $v_1vv_3$ is $(L:3)$-colorable,
there exists $\alpha\in (L(v)\cup L(v_1))\setminus L(v_3)$.  Choose $\beta\in L(v_2)\setminus (\{\alpha\}\cup L(u_2))$
and distinct colors $\gamma,\gamma'\in L(v_2)\setminus L(v)$ arbitrarily.  Let $\varphi(v_2)$ be a $3$-element subset of
$L(v_2)$ containing $\beta$, $\gamma$, and $\gamma'$.  Note that $|L(u_2)\setminus\varphi(v_2)|\ge 3$ by the choice of $\beta$,
and thus we can choose $\varphi(u_2)$ as a $3$-element subset of $L(u_2)\setminus\varphi(v_2)$.
By the choice of $\gamma$ and $\gamma'$, there exists a $4$-element subset $A$ of $L(v)$, containing $\alpha$ if $\alpha\in L(v)$.
Choose distinct colors $\kappa,\kappa'\in L(v_1)\setminus A$, such that $\kappa=\alpha$ if $\alpha\not\in A$.
Let $\varphi(u_1)$ be a $3$-element subset of $L(u_1)\setminus \{\kappa,\kappa'\}$, and let $B=L(v_1)\setminus \varphi(u_1)$.
Note that $|B|\ge 5$, $|A\cup B|\ge |A\cup \{\kappa,\kappa'\}|=6$, and $|A\cup B\cup L(v_3)|\ge |L(v_3)\cup \{\alpha\}|=9$,
and thus by Lemma~\ref{triangle}, $\varphi$ extends to an $(L:3)$-coloring of $H$.

Suppose now that $v_1v_3\not\in E(H)$.  Since the path $u_1v_1vv_3$ is $(L:3)$-colorable, we have
$L(u_1)\neq L(v_1)\neq L(v)\neq L(v_3)$, and furthermore, if $|L(v_1)\setminus L(u_1)|=1$, then
$L(v_1)\setminus L(u_1)\neq L(v)\setminus L(v_3)$.  Hence, there exists a color $\alpha\in L(v)\setminus L(v_3)$
such that $|L(v_1)\setminus (\{\alpha\}\cup L(u_1))|\ge 1$.  Let $\beta$ be any color in $L(v_1)\setminus (\{\alpha\}\cup L(u_1))$.
If $\alpha\in L(v_1)$, then let $\alpha'$ be any color in $L(v)\setminus L(v_1)$, otherwise let $\alpha'$ be any color in $L(v)\setminus \{\alpha,\beta\}$.
If $\beta\in L(v)$, then let $\beta'$ be any color in $L(v_1)\setminus L(v)$, otherwise let $\beta'$ be any color in $L(v_1)\setminus\{\alpha,\alpha',\beta\}$.
Let $\gamma$ be any color in $L(u_1)\setminus L(v_1)$, let $\varepsilon$ be any color in $L(v_3)\setminus L(v)$, and let $\kappa$ be any color in $L(v_2)\setminus (\{\alpha,\alpha'\}\cup L(u_2)\})$.
Let $f'(u_1)=f'(v_3)=f'(v_2)=2$, $f'(v_1)=f'(v)=1$, and $f'(u_2)=3$.
By Lemma~\ref{lemma-redulist}, it suffices to prove that $H$ has
an $(L':f')$-coloring for any list assignment $L'$ such that $|L'(u_1)|=|L'(v_3)|=3$, $|L'(v_1)|=2$, $|L'(v)|=1$, and $|L'(v_2)|=|L'(u_2)|=5$.
This is the case by coloring the vertices of $H$ greedily in order $v$, $v_3$, $v_1$, $u_1$, $v_2$, and $u_2$.
\end{proof}

\section{Properties of a minimal counterexample}

We are going to prove a mild strengthening of Theorem~\ref{MTl} where a clique (one vertex, two adjacent
vertices, or a triangle) is precolored.
A (hypothetical) \emph{counterexample} (to this strengthening) is a triple $(G,L,Z)$, where
$G$ is a plane graph without $4$- or $5$-cycles, $Z$ is the vertex set of a clique of $G$, and $L$ is an assignment
of lists of size $11$ to vertices of $V(G)\setminus Z$ and pairwise disjoint lists of size $3$ to vertices $Z$, such that
$G$ is not $(L:3)$-colorable.  The \emph{order} of the counterexample is the number of vertices of $G$.
A counterexample is \emph{minimal} if there exists no counterexample of smaller order.

\begin{lemma}\label{conn}
If $(G,L,Z)$ is a minimal counterexample, then $G$ is $2$-connected and every triangle in $G$ bounds a face.
\end{lemma}
\begin{proof}
If $G$ is not $2$-connected, then there exist proper induced subgraphs $G_1$ and $G_2$ of $G$
and a vertex $z\in V(G_2)$ such that $G=G_1\cup G_2$, $V(G_1\cap G_2)\subseteq\{z\}$, and
$Z\subseteq V(G_1)$.  By the minimality of the counterexample, there exists an $(L:3)$-coloring $\varphi_1$ of $G_1$.
If $z\in V(G_1)$, then let $L'(z)=\varphi_1(z)$, otherwise let $L'(z)$ be any $3$-element subset of $L(z)$.
Let $L'(v)=L(v)$ for all $v\in V(G_2)\setminus \{z\}$.  Note that $(G_2,L',\{z\})$ is not a counterexample since it has smaller
order than $(G,L,Z)$,
and thus there exists an $(L':3)$-coloring $\varphi_2$ of $G_2$.  However, then $\varphi_1$ and $\varphi_2$ combine
to an $(L:3)$-coloring of $G$, which is a contradiction.

Similarly, if $G$ contains a non-facial triangle $T$, then there exist proper induced subgraphs $G_1$ and $G_2$ of $G$
such that $G=G_1\cup G_2$, $T=G_1\cap G_2$, and $Z\subseteq V(G_1)$.  By the minimality of the counterexample, there exists an $(L:3)$-coloring $\varphi_1$ of $G_1$.
Let $L'(z)=\varphi_1(z)$ for all $z\in V(T)$ and $L'(v)=L(v)$ for all $v\in V(G_2)\setminus V(T)$.  Note that $(G_2,L',V(T))$ is not a counterexample since it has smaller order than $(G,L,Z)$,
and thus there exists an $(L':3)$-coloring $\varphi_2$ of $G_2$.  However, then $\varphi_1$ and $\varphi_2$ combine
to an $(L:3)$-coloring of $G$, which is a contradiction.
\end{proof}

Let $L$ be a list assignment for a graph $G$, let $H$ be an induced subgraph of $G$,
and let $\psi$ be an $(L:3)$-coloring of $G-V(H)$.  Let $L_\psi$
denote the list assignment for $H$ defined by
$$L_\psi(v)=L(v)\setminus\bigcup_{uv\in E(G), u\not\in V(H)} \psi(u)$$
for all $v\in V(H)$.  Note that
\begin{equation}\label{eq:size}
|L_\psi(v)|\ge |L(v)|-3(\deg_G(v)-\deg_H(v)).
\end{equation}
Furthermore, any $(L_\psi:3)$-coloring of $H$ combines with $\psi$ to an $(L:3)$-coloring of $G$.
Hence, the following claim holds.
\begin{proposition}\label{obs:extend}
Let $(G,L,Z)$ be a minimal counterexample and let $H$ be an induced subgraph of $G$ disjoint from $Z$.
If $\psi$ is an $(L:3)$-coloring of $G-V(H)$, then $H$ is not $(L_\psi:3)$-colorable.
\end{proposition}

In a counterexample $(G,L,Z)$, a vertex $v\in V(G)$ is \emph{internal} if $v\not\in Z$.

\begin{lemma}\label{MD}
If $(G,L,Z)$ is a minimal counterexample, then every internal vertex of $G$ has degree at least $3$.
\end{lemma}
\begin{proof}
Suppose for a contradiction that there exists a vertex $v\in V(G)\setminus Z$ of degree at most two.
By the minimality of the counterexample, the graph $G-v$ has an $(L:3)$-coloring $\psi$.
By Proposition~\ref{obs:extend}, $v$ is not $(L_\psi:3)$-colorable.  However, this is a contradiction,
since $|L_\psi(v)|\ge 11-2\cdot 3=5$ by (\ref{eq:size}).
\end{proof}

\begin{lemma}
\label{33-path}
Let $(G,L,Z)$ be a minimal counterexample.  Let $P=v_1\ldots v_k$ be a path in $G$ disjoint from $Z$ such that
$3\le k\le 6$, $\deg (v_1) = \deg (v_2) = \deg (v_k) = 3$ and $\deg(v_i)=4$ for $3\le i\le k-1$.
Then $k=3$ and $v_1v_3\in E(G)$.
\end{lemma}
\begin{proof}
Suppose for a contradiction that either $k\ge 4$, or $k=3$ and $v_1v_3\not\in E(G)$.
Choose such a path $P$ with $k$ minimum.

Note that $G$ contains at most one of the edges $v_1v_k$ and $v_2v_k$; this follows by the assumptions
if $k=3$ and by the fact that $G$ does not contain $4$- or $5$-cycles otherwise.  By the minimality of $k$,
we conclude that $G$ contains neither of these edges, with the exception of the case $k=3$ and the edge $v_2v_3$
(since otherwise we can consider a path $v_1v_2v_k$ or $v_2v_1v_k$ instead of $P$).
Consequently, by the minimality of $k$, it follows that the path $P$ is induced.

By the minimality of $G$, the graph $G-\{v_1,v_2\}$ has an $(L:3)$-coloring $\psi_0$.
Let $\psi$ be the restriction of $\psi_0$ to $G-V(P)$, and consider the list assignment $L_\psi$ for $P$.
By the existence of $\psi_0$, we conclude that $P-\{v_1,v_2\}$ is $(L_\psi:3)$-colorable.
Note that $|L_\psi(v_2)|\ge 8$ and $|L_\psi(v_i)|\ge 5$ for $i=1,\ldots, k$ by (\ref{eq:size}).
By Lemmas~\ref{3-3-3}, \ref{3-3-4-3}, \ref{3-3-4-4-3}, and \ref{3-3-4-4-4-3},
the path $P$ is $(L_\psi:3)$-colorable.  However, this contradicts Proposition~\ref{obs:extend}.
\end{proof}

\begin{lemma}\label{6-cycle}
Let $(G,L,Z)$ be a minimal counterexample.  Let $C=v_1\ldots v_6$ be a $6$-cycle in $G$ disjoint from $Z$ such that
all vertices of $C$ have degree at most $4$.  Then at most one vertex of $C$ has degree three.
\end{lemma}

\begin{proof}
Suppose for a contradiction that $C$ contains at least two vertices of degree three,
and let $S$ be the set of two such vertices.  Note that since $G$ does not contain $4$- or $5$-cycles,
the cycle $C$ is induced.

By the minimality of $G$, the graph $G-S$ has an $(L:3)$-coloring $\psi_0$.
Let $\psi$ be the restriction of $\psi_0$ to $G-V(C)$, and consider the list assignment $L_\psi$ for $C$.
By the existence of $\psi_0$, we conclude that $C-S$ is $(L_\psi:3)$-colorable.
Note that $|L_\psi(v)|\ge 8$ for all $v\in S$ and $|L_\psi(v_i)|\ge 5$ for $v\in V(C)\setminus S$ by (\ref{eq:size}).
By Lemma~\ref{l6cycle}, the cycle $C$ is $(L_\psi:3)$-colorable.  However, this contradicts Proposition~\ref{obs:extend}.
\end{proof}

\begin{lemma}\label{tria3}
Let $(G,L,Z)$ be a minimal counterexample.  Let $v_1v_2v_3$ be a triangle in $G$ disjoint from $Z$ such that $\deg(v_1),\deg(v_3)\le 4$
and $\deg(v_2)=3$.  Then $v_3$ has no neighbor $v_4\not\in \{v_1,v_2\}\cup Z$ of degree $3$.
\end{lemma}
\begin{proof}
Suppose for a contradiction that $v_3$ has such a neighbor $v_4$.
Note that $v_1v_4, v_2v_4\not\in E(G)$, since $G$ does not contain $4$-cycles.
Let $H$ be the subgraph of $G$ induced by $\{v_1,v_2,v_3,v_4\}$.

By the minimality of $G$, the graph $G-v_4$ has an $(L:3)$-coloring $\psi_0$.
Let $\psi$ be the restriction of $\psi_0$ to $G-V(H)$, and consider the list assignment $L_\psi$ for $H$.
By the existence of $\psi_0$, we conclude that the triangle $v_1v_2v_3$ is $(L_\psi:3)$-colorable.
Note that $|L_\psi(v_2)|,|L_\psi(v_3)|\ge 8$ and $|L_\psi(v_1)|,|L_\psi(v_4)|\ge 5$ by (\ref{eq:size}).
By Lemma~\ref{lollipop}, the graph $H$ is $(L_\psi:3)$-colorable.  However, this contradicts Proposition~\ref{obs:extend}.
\end{proof}

\begin{lemma}
\label{34-path}
Let $(G,L,Z)$ be a minimal counterexample.  Then $G$ does not contain a path $P=v_1\ldots v_k$ disjoint from $Z$ such that
$5\le k\le 7$, $\deg (v_1) = \deg (v_3) = \deg (v_k) = 3$, $\deg(v_2)=4$, and $\deg(v_i)=4$ for $4\le i\le k-1$.
\end{lemma}
\begin{proof}
Suppose for a contradiction that $G$ contains such a path $P$.
By Lemma~\ref{33-path}, the vertices $v_1$, $v_3$, and $v_k$ form an independent set.
By considering $P$ as short as possible, we can assume that $v_3\ldots v_k$ is an induced path.
Note that $v_2v_j\not\in E(G)$ for $5\le j\le k$ and $v_1v_i\not\in E(G)$ for $4\le i\le k-1$ by the absence of $4$- and $5$-cycles and by Lemma~\ref{6-cycle}.
Furthermore, $v_2v_4\not\in E(G)$ by Lemma~\ref{tria3}.
Hence, $P$ is an induced path.

By the minimality of $G$, the graph $G-v_3$ has an $(L:3)$-coloring $\psi_0$.
Let $\psi$ be the restriction of $\psi_0$ to $G-V(P)$, and consider the list assignment $L_\psi$ for $P$.
By the existence of $\psi_0$, we conclude that $P-v_3$ is $(L_\psi:3)$-colorable.
Note that $|L_\psi(v_3)|\ge 8$ and $|L_\psi(v_i)|\ge 5$ for $1\le i\le k$ by (\ref{eq:size}).
By Lemma~\ref{3-4-3---3}, the path $P$ is $(L_\psi:3)$-colorable.  However, this contradicts Proposition~\ref{obs:extend}.
\end{proof}

We now consider the neighborhoods of vertices of degree 4.

\begin{lemma}
\label{4-vertex}
Let $(G,L,Z)$ be a minimal counterexample and let $v$ be an internal vertex of $G$ of degree four.
Then $v$ has at most two internal neighbors of degree three.
\end{lemma}
\begin{proof}
Suppose for a contradiction that $v$ has three such neighbors $v_1, v_2,v_3\in V(G)\setminus Z$.
Note that $\{v_1,v_2,v_3\}$ is an independent set by Lemma~\ref{tria3}.
By the minimality of $G$, the graph $G-\{v,v_1,v_2,v_3\}$ has an $(L:3)$-coloring $\psi$.
Note that $|L_\psi(v)|\ge 8$ and $|L_\psi(v_i)|\ge 5$ for $1\le i\le 3$ by (\ref{eq:size}).
By Lemma~\ref{claw}, the subgraph $G[\{v,v_1,v_2,v_3\}]$ is $(L_\psi:3)$-colorable.  However, this contradicts Proposition~\ref{obs:extend}.
\end{proof}

Next, let us consider the neighborhoods of vertices of degree 5.

\begin{lemma}
\label{5-vertex}
Let $(G,L,Z)$ be a minimal counterexample and let $v$ be an internal vertex of $G$ of degree five.
If $v$ has four internal neighbors $v_1$, \ldots, $v_4$ of degree three, then
$G[\{v_1,\ldots,v_4\}]$ is a perfect matching.
\end{lemma}
\begin{proof}
Suppose for a contradiction that $G[\{v_1,\ldots,v_4\}]$ is not a perfect matching.
Since $G$ does not contain $4$-cycles, it follows that $G[\{v_1,\ldots,v_4\}]$ has at most one edge;
we can assume that it contains no edge other than $v_3v_4$.  Let $H=G[\{v,v_1,v_2,v_3,v_4\}]$.

By the minimality of $G$, the graph $G-\{v_1,v_2\}$ has an $(L:3)$-coloring $\psi_0$.
Let $\psi$ be the restriction of $\psi_0$ to $G-V(H)$, and consider the list assignment $L_\psi$ for $H$.
By the existence of $\psi_0$, we conclude that $H-\{v_1,v_2\}$ is $(L_\psi:3)$-colorable. Note that $|L_\psi(v)|\ge 8$, $|L_\psi(v_i)|\ge 5$ for $1\leq i\leq 2$ and  $|L_\psi(v_i)|\ge 2+3\deg_H(v_i)$ for $3\leq i\leq 4$ by (\ref{eq:size}).
By Lemma~\ref{claw5}, the graph $H$ is $(L_\psi:3)$-colorable.  However, this contradicts Proposition~\ref{obs:extend}.
\end{proof}

\begin{lemma}
\label{5-vertex3}
Let $(G,L,Z)$ be a minimal counterexample and let $v$ be an internal vertex of $G$ of degree five.
If $v$ has three internal neighbors $v_1$, $v_2$, and $v_3$ of degree three, then
$v_1$ has no neighbor of degree three not belonging to $Z$ and not adjacent to $v$.
\end{lemma}
\begin{proof}
Suppose for a contradiction that $v_1$ has a neighbor $u_1\not\in Z\cup N_G(v)$ of degree three.
Since $G$ does not contain $4$-cycles, it follows that $u_1v_2,u_1v_3\not\in V(G)$ and that
$G$ contains at most one of the edges $v_1v_2$, $v_2v_3$, and $v_1v_3$.
Let $H=G[\{v,v_1,v_2,v_3,u_1\}]$.

By the minimality of $G$, the graph $G-v_1$ has an $(L:3)$-coloring $\psi_0$.
Let $\psi$ be the restriction of $\psi_0$ to $G-V(H)$, and consider the list assignment $L_\psi$ for $H$.
By the existence of $\psi_0$, we conclude that $H-v_1$ is $(L_\psi:3)$-colorable.
Note that $|L_\psi(v)|\ge 5$, $|L_\psi(u_1)|\ge 5$ and  $|L_\psi(v_i)|\ge 2+3\deg_H(v_i)$ for $1\leq i\leq 3$ by (\ref{eq:size}).
By Lemma~\ref{claw53}, the graph $H$ is $(L_\psi:3)$-colorable.  However, this contradicts Proposition~\ref{obs:extend}.
\end{proof}

\begin{lemma}\label{5-vertex43}
Let $(G,L,Z)$ be a minimal counterexample, and let $P=u_1v_1vv_2u_2$ be a path in $G$ vertex-disjoint from
$Z$.  If $vu_2\not\in E(G)$, $\deg(v)=5$, $\deg(u_1)=\deg(u_2)=\deg(v_2)=3$, and $\deg(v_1)=4$, then $v$ has no internal neighbors
of degree three distinct from $v_2$ and $u_1$.
\end{lemma}
\begin{proof}
Suppose for a contradiction that $v$ has a neighbor $v_3\not\in \{v_2,u_1\}\cup Z$ of degree three.
Note that $G$ does not contain the edge $v_1v_2$ by Lemma~\ref{33-path}
and the edge $vu_1$ by Lemma~\ref{5-vertex3}.
Since $G$ does not contain $4$- or $5$-cycles, $P$ is an induced path.
By Lemma~\ref{33-path} and absence of $4$- and $5$-cycles, $v_3$ has no neighbors among $u_1$, $v_2$, and $u_2$.
Consequently, since $G$ does not contain $4$- or $5$-cycles,
$H=G[\{u_1,v_1,v,v_2,u_2,v_3\}]$ consists of the path $P$, the edge $vv_3$, and possibly the edge $v_1v_3$.

By the minimality of $G$, the graph $G-v_2$ has an $(L:3)$-coloring $\psi_0$.
Let $\psi$ be the restriction of $\psi_0$ to $G-V(H)$, and consider the list assignment $L_\psi$ for $H$.
By the existence of $\psi_0$, we conclude that $H-v_2$ is $(L_\psi:3)$-colorable. Note that $|L_\psi(v)|\ge 5$, $|L_\psi(u_i)|\ge 5$ for $1\leq i\leq 2$, $|L_\psi(v_1)|\ge 3\deg_H(v_1)-1$, $|L_\psi(v_2)|\ge 8$ and  $|L_\psi(v_3)|\ge 2+3\deg_H(v_3)$ by (\ref{eq:size}).
By Lemma~\ref{claw543}, the graph $H$ is $(L_\psi:3)$-colorable.  However, this contradicts Proposition~\ref{obs:extend}.
\end{proof}

\section{Discharging}\label{sec-discharge}

\subsection{Notation}
Consider a minimal counterexample $(G,L,Z)$.
We say that the faces of $G$ of length at least $6$ are \emph{$6^+$-faces}.
Since $G$ is $2$-connected by Lemma~\ref{conn}, every face of $G$ is bounded by a cycle,
and in particular, every face of $G$ is either a $3$-face or a $6^+$-face.
A vertex $v\in V(G)$ is a \emph{$k$-vertex} if $v$ is internal and $\deg(v)=k$.
We say that $v$ is a \emph{$k^+$-vertex} if either $v\in Z$ or $\deg(v)\ge k$.  

Let $v_1vv_2$ be a part of the cycle bounding a $6^+$-face $f$ of $G$,
and for $i\in\{1,2\}$, let $f_i\neq f$ be the face incident with the edge $vv_i$.
If both $f_1$ and $f_2$ are $3$-faces, we say that $v$ is \emph{type-II incident} with $f$.
If exactly one of $f_1$ and $f_2$ is a $3$-face, we say that $v$ is \emph{type-I incident} with $f$.
If neither $f_1$ nor $f_2$ is a $3$-face, we say that $v$ is \emph{type-0 incident} with $f$.
See Figure~\ref{fig-incid} for an illustration.

\begin{figure}[!htb]
\centering
{\includegraphics[height=0.27\textwidth]{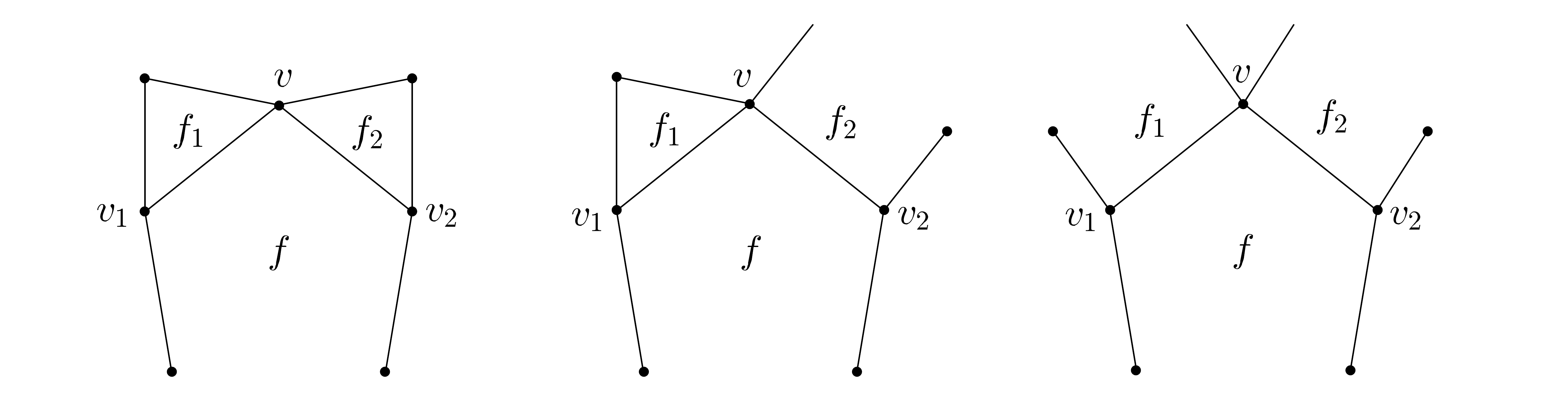}}

\caption{Type-II, type-I, and type-0 incidences.}\label{fig-incid}
\end{figure}

Suppose that in the situation described in the previous paragraph,
$v$ is a 4-vertex type-I incident with $f$, where $f_1$ is a $3$-face,
$v_1$ is a $4^+$-vertex and $v_2$ is a $5^+$-vertex.
Let $v_2vx$ be the subpath of the cycle bounding $f_2$ centered at $v$.
If $x$ is a $3$-vertex, then we say $v$ is {\em type-I-1 incident} with $f$.
If $x$ is a $4$-vertex, $f_2$ is bounded by a $6$-cycle $xvv_2w_1w_2w_3$,
$w_1$ and $w_3$ are $3$-vertices and  $w_2$ is a $4$-vertex type-II incident with $f_2$,
then we say $v$ is {\em type-I-2 incident} with $f$.  See Figure~\ref{fig-type-Ia} for an illustration.

\begin{figure}[!htb]
\centering
{\includegraphics[height=0.31\textwidth]{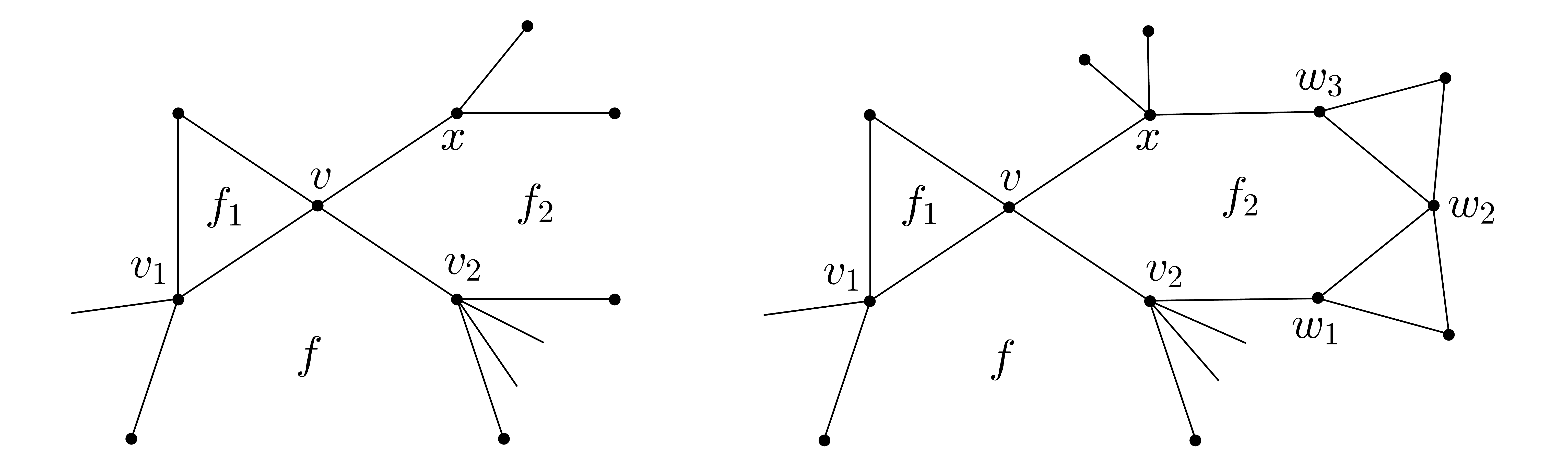}}

\caption{Type-I-1 and type-I-2 incidences.}\label{fig-type-Ia}
\end{figure}

Let $v_0v_1vv_2v_3$ be a subpath of the cycle bounding a $6^+$-face $f$,
where $vv_1$ is incident with a $3$-face, $vv_2$ is not incident with a $3$-face,
$v$ is a $5$-vertex and $v_1$ is a $3$-vertex.  Let $v_1$, $v_2$, $x_1$, $x_2$, $x_3$ be the neighbors of $v$
listed in cyclic order according to their drawing in $G$.
If both $v_2$ and $v_3$ are $3$-vertices, then we say $v$ is {\em type-I-3 incident} with $f$.
If $v_0$ and $x_1$ are $3$-vertices and $x_1$ is contained in a triangle $x_1yz$
for $4^+$-vertices $y$ and $z$ distinct from $x_2$ and $x_3$,
then we say $v$ is {\em type-I-4 incident} with $f$.
See Figure~\ref{fig-type-Ib} for an illustration.

\begin{figure}[!htb]
\centering
{\includegraphics[height=0.39\textwidth]{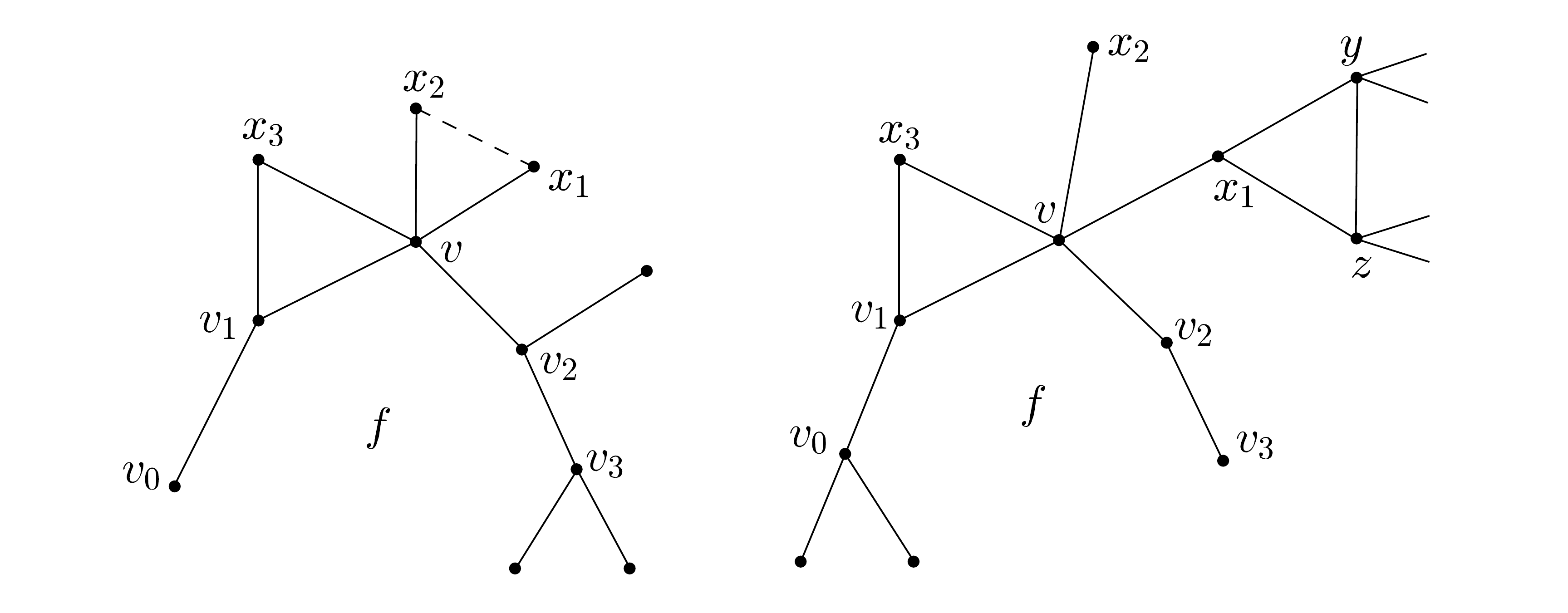}}

\caption{Type-I-3 and type-I-4 incidences.}\label{fig-type-Ib}
\end{figure}

Let $v_1vv_2v_3$ be a part of the cycle bounding a $6^+$-face $f$,
where $v$ is a $5$-vertex type-0 incident with $f$.  Let $v_1$, $v_2$, $x_1$, $x_2$, $x_3$ be the neighbors of $v$
listed in cyclic order according to their drawing in $G$.
If $v_1$ and $v_2$ are $3$-vertices, then we say $v$ is {\em type-0-1 incident} with $f$.
If $v_2$ and $v_3$ are $3$-vertices, then we say $v$ is {\em type-0-2 incident} with $f$.
If both $x_1$ and $x_3$ belong to triangles containing only $3$-vertices distinct from $x_2$,
then we say $v$ is {\em type-0-3 incident} with $f$.
See Figure~\ref{fig-type-0} for an illustration.

\begin{figure}[!htb]
\centering
{\includegraphics[height=0.35\textwidth]{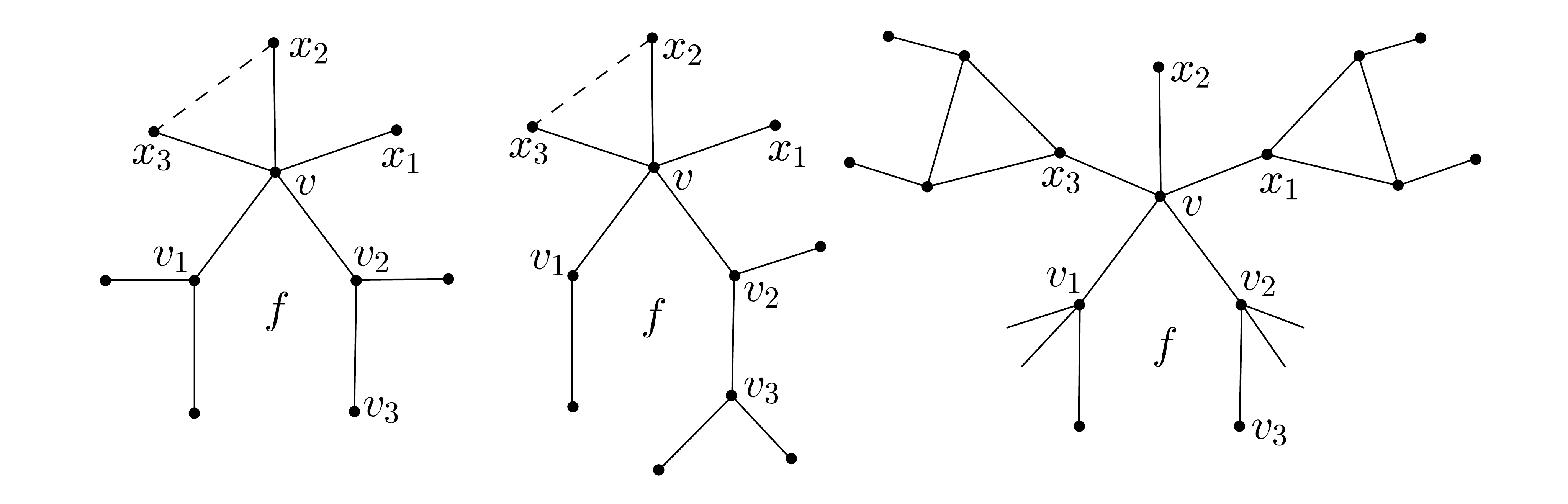}}
\caption{Type-0-1, type-0-2, and type-0-3 incidences.}\label{fig-type-0}
\end{figure}

\subsection{Initial charge and discharging rules}

Now we proceed by the discharging method.
Consider a minimal counterexample $(G,L,Z)$.
Set the initial charge of every vertex $v$ of $G$ to be $\ch_0(v)=2\deg (v)-6$,
and the initial charge of every face $f$ of $G$ to be $\ch_0(f)=|f|-6$.
By Euler's formula,
\begin{align}
\sum_{v\in V(G)}\ch_0(v)+\sum_{f\in F(G)}\ch_0(f)&=\sum_{v\in V(G)}(2\deg (v)-6)+\sum_{f\in F(G)}(|f|-6)\nonumber\\
&=6(|E(G)|-|V(G)|-|F(G)|)=-12.\label{eq:sum}
\end{align}
We redistribute the charges according to the following rules:

\begin{description}
\item[Rt] If a $6^+$-face $f$ shares an edge with a $3$-face $f'$, then $f$ sends 1 to $f'$.
\item[R4] Suppose $v$ is a 4-vertex and $f$ is a $6^+$-face incident with $v$.
\begin{itemize}
\item[(II)] If $v$ is type-II incident with $f$, then $v$ sends $1$ to $f$.
\item[(I)] Suppose $v$ is type-I incident with $f$. If $v$ is type-I-1 or type-I-2 incident with $f$,
then $v$ sends $1/2$ to $f$, otherwise $v$ sends $1$ to $f$.
\item[(0)] Suppose $v$ is type-0 incident with $f$.  If either $v$ is not incident with any $3$-faces
or $v$ is type-I-1 or type-I-2 incident with another $6^+$-face, then
$v$ sends $1/2$ to $f$.
\end{itemize}
\item[R5] Suppose $v$ is a 5-vertex and $f$ is a $6^+$-face incident with $v$.
\begin{itemize}
\item[(II)] Suppose $v$ is type-II incident with $f$.  If $v$ is type-I-3 incident with another $6^+$-face,
then $v$ sends $1$ to $f$, otherwise $v$ sends $2$ to $f$.
\item[(I)] Suppose $v$ is type-I incident with $f$.  If $v$ is type-I-3 or type-I-4 incident with $f$,
then $v$ sends $3/2$ to $f$, otherwise $v$ sends $1$ to $f$.
\item[(0)] Suppose $v$ is type-0 incident with $f$.  If $v$ is type-0-1 or type-0-2 incident with $f$,
then $v$ sends $1$ to $f$; otherwise, if $v$ is not type-0-3 incident with $f$, then $v$ sends $1/2$ to $f$.
\end{itemize}
\item[R6] Suppose $v$ is a $6^+$-vertex and $f$ is a $6^+$-face incident with $v$.
\begin{itemize}
\item[(II)] If $v$ is type-II incident with $f$, then $v$ sends $2$ to $f$.
\item[(I)] If $v$ is type-I incident with $f$, then $v$ sends $3/2$ to $f$.
\item[(0)] If $v$ is type-0 incident with $f$, then $v$ sends $1$ to $f$.
\end{itemize}
\end{description}
In the situations of rules R4, R5, and R6, we write $\ch(v\to f)$ for the amount of
charge sent from $v$ to $f$.

\subsection{Final charges of vertices}

Let $\ch$ denote the charge assignment after performing the charge redistribution
using the rules Rt, R4, R5, and R6.

\begin{lemma}\label{charge-4vertex}
Let $(G,L,Z)$ be a minimal counterexample and let $v$ be a vertex of $G$.
If $v$ is a $4$-vertex, then $\ch(v)\ge 0$.
\end{lemma}
\begin{proof}
Note that $\ch_0(v)=2$. Let $v_1$, \ldots, $v_4$ be the neighbors of $v$ listed in cyclic order according to their drawing in $G$.
For $1\le i\le 4$, let $f_i$ be the face whose boundary contains the path $v_ivv_{i+1}$ (where $v_5=v_1$).
Since $G$ contains no $4$-cycles, we can without loss of generality assume that $f_2$ and $f_4$ are $6^+$-faces.
If $f_1$ and $f_3$ are $3$-faces, then $v$ is type-II incident with them and $\ch(v\to f_2)=\ch(v\to f_4)=1$ by R4(II),
and $\ch(v)=2-2\times 1=0$.  Hence, suppose that $f_3$ is a $6^+$-face.

Suppose now that $f_1$ is a $3$-face.  If $v$ neither type-I-1 nor type-I-2 incident with $f_2$ and $f_4$,
then $\ch(v\to f_3)=0$ by R4(0) and $\ch(v\to f_2)=\ch(v\to f_4)=1$ by R4(I), and
$\ch(v)=2-2\times 1=0$.  If $v$ is type-I-1 or type-I-2 incident with say $f_2$, then
$\ch(v\to f_3)=1/2$ by R4(0) and $\ch(v\to f_2)=1/2$ and $\ch(v\to f_4)\le 1$ by R4(I),
and $\ch(v)\ge 2-1-2\times 1/2=0$.

Finally, if $v$ is not incident with any $3$-faces, then $\ch(v\to f_i)=1/2$ by R4(0) for $1\le i\le 4$
and $\ch(v)=2-4\times 1/2=0$.
\end{proof}

\begin{lemma}\label{charge-5vertex}
Let $(G,L,Z)$ be a minimal counterexample and let $v$ be a vertex of $G$.
If $v$ is a $5$-vertex, then $\ch(v)\ge 0$.
\end{lemma}
\begin{proof}
Note that $\ch_0(v)=4$. Let $v_1$, \ldots, $v_5$ be the neighbors of $v$ listed in cyclic order according to their drawing in $G$.
For $1\le i\le 5$, let $f_i$ be the face whose boundary contains the path $v_ivv_{i+1}$ (where $v_6=v_1$).
Since $G$ contains no $4$-cycles, we can without loss of generality assume that $f_2$, $f_4$, and $f_5$
are $6^+$-faces.

Suppose first that $f_1$ and $f_3$ are $3$-faces.  Then $v$ is type-II incident with $f_2$ and
type-I incident with $f_4$ and $f_5$.  Note that $v$ cannot be type-I-4 incident with $f_4$ and $f_5$.
If $v$ is not type-I-3 incident with $f_4$ and $f_5$,
then $\ch(v\to f_2)=2$ by R5(II) and $\ch(v\to f_4)=\ch(v\to f_5)=1$ by R5(I), and thus $\ch(v)=4-2-2\times 1=0$.
If $v$ is type-I-3 incident with $f_4$ or $f_5$, then $\ch(v\to f_2)=1$ by R5(II) and $\ch(v\to f_4),\ch(v\to f_5)\le 3/2$
by R5(I), and thus $\ch(v)\ge 4-1-2\times 3/2=0$.

Hence, we can assume that $f_3$ is a $6^+$-face.  Suppose now that $f_1$ is a $3$-face.
If $v$ is type-I-3 or type-I-4 incident with neither $f_2$ nor $f_5$, then $\ch(v\to f_i)\le 1$ by
R5(I) and R5(0) for $2\le i\le 5$, and $\ch(v)\ge 4-4\times 1=0$.  If $v$ is type-I-3 incident with $f_2$,
then $v_4$, $v_5$, and $v_1$ are $4^+$-vertices by Lemma~\ref{5-vertex3},
and thus $v$ is neither type-I-3 nor type-I-4 incident with $f_5$ and $v$ is neither type-0-1 nor type-0-2
incident with $f_4$.
If $v$ is type-I-4 incident with $f_2$, then $v_3$, $v_5$, and $v_1$ are $4^+$-vertices by Lemma~\ref{5-vertex3},
and thus $v$ is neither type-I-3 nor type-I-4 incident with $f_5$ and $v$ is neither type-0-1 nor type-0-2
incident with $f_3$ and $f_4$.
In either case, $\ch(v\to f_2)=3/2$ and $\ch(v\to f_5)=1$ by R5(I) and $\ch(v\to f_4)\le 1/2$
and $\ch(v\to f_3)\le 1$ by R5(0), and $\ch(v)\ge 4-3/2-2\times 1-1/2=0$.

Finally, let us consider the case that $v$ is incident with no $3$-faces.  If $v$ is type-0-1 or type-0-2
incident with at most three faces, then $\ch(v)\ge 4-3\times 1-2\times 1/2=0$ by R5(0).
Hence, suppose that $v$ is type-0-1 or type-0-2 incident with at least 4 faces.
By Lemma~\ref{5-vertex}, $v$ is adjacent to at most three $3$-vertices.  If
$v$ is adjacent to three $3$-vertices, then by Lemma~\ref{5-vertex3} $v$ is not type-0-2 incident
with any faces, and clearly $v$ is type-0-1 incident with at most two faces, which is a contradiction.
Hence, $v$ is adjacent to at most two $3$-vertices, and by symmetry, we can assume that $v_5$ and $v_1$
are $4^+$-vertices.  Then $v$ is neither type-0-1 nor type-0-2 incident with $f_5$, and thus
it is type-0-1 or type-0-2 incident with $f_1$, \ldots, $f_4$.  It cannot be type-0-1 incident with $f_1$
and $f_4$, and thus it is type-0-2 incident with these faces; i.e., $v_2$ and $v_4$ are $3$-vertices
and have $3$-vertex neighbors $x_2$ and $x_4$ incident with $f_1$ and $f_4$.  By Lemma~\ref{5-vertex3}, $v_3$ is a
$4^+$-vertex.  Consequently, $v$ is also type-0-2 incident with $f_2$ and $f_3$, and thus
$v_2$ and $v_4$ have $3$-vertex neighbors $x'_2$ and $x'_4$ incident with $f_2$ and $f_3$.
By Lemma~\ref{33-path}, $x_2x'_2\in E(G)$ and $x_4x'_4\in E(G)$.  But then $v$ is type-0-3 incident with $f_5$
and $\ch(v\to f_5)=0$ by R5(0); and thus $\ch(v)=4-4\times 1=0$.
\end{proof}

\begin{lemma}
\label{vertex}
Let $(G,L,Z)$ be a minimal counterexample and let $v$ be a vertex of $G$.
If $v$ is internal, then $\ch(v)\ge 0$.  If $v\in Z$, then $\ch(v)=\deg(v)-6$.
\end{lemma}
\begin{proof}
By Lemma~\ref{MD}, if $v$ is internal then $\deg (v)\geq 3$.
If $v$ is a $3$-vertex, then $\ch(v)=\ch_0(v)=0$.
If $v$ is a $4$- or $5$-vertex, then $\ch(v)\ge 0$ by Lemmas~\ref{charge-4vertex} and \ref{charge-5vertex}.

Hence, suppose that $v$ is a $6^+$-vertex, incident with $t$ $3$-faces,
and type-II, type-I and type-0 incident with $d_{II}$, $d_I$, and $d_0$ $6^+$-faces, respectively.
Note that $d_{II}+d_I/2=t$.  By R6, we have
\begin{align*}
\ch(v)&=\ch_0(v)-d_{II}\times 2-d_I\times 3/2-d_0\\
&=\ch_0(v)-(d_{II}+d_I+d_0+t)=\ch_0(v)-\deg(v)=\deg(v)-6.
\end{align*}
If $v$ is internal, then $\deg(v)\ge 6$, and thus $\ch(v)\ge 0$.
\end{proof}

\subsection{Final charge of faces}

Let $f$ be a $6^+$-face.
A subpath $S=u_0u_1\ldots u_t$ of the cycle bounding $f$ with at least two vertices
is called a \emph{segment} of $f$
if $u_1$, \ldots, $u_{t-1}$ are type-II incident with $f$, and $u_0$ and $u_t$
are type-I incident with $f$.
In particular, for $1\le i\le t$, the edge $u_{i-1}u_i$ is incident with a $3$-face.
Note that the segments of $f$ are pairwise vertex-disjoint.
Let us define $\ch(S)=-t+\sum_{i=0}^t \ch(u_i\to f)$;
note that $\ch(S)$ denotes the amount of charge received by $f$ from vertices of
the segment, minus the amount sent by the rule Rt to $3$-faces incident with edges of $S$.
If $\ch(S)<0$, then we say $S$ is a \emph{negative segment}.
A \emph{$t$-segment} is a segment with $t$ edges.

\begin{proposition}
\label{segment}
Let $(G,L,Z)$ be a minimal counterexample.  Let $S=u_0\ldots u_t$ be a negative $t$-segment of a $6^+$-face $f$ of $G$,
where $\deg(u_0)\le \deg(u_t)$.
Then all vertices of $S$ are internal, and either
\begin{itemize}
\item both $u_0$ and $u_t$ are $3$-vertices and $\ch(S)=-1$, or
\item $t\ge 2$, $u_0$ is a $3$-vertex, $u_t$ is a $4$-vertex type-I-1 or type-I-2 incident with $f$, and $\ch(S)=-1/2$.
\end{itemize}
Furthermore, if either $t\le 3$, or $t\le 5$ and $\ch(S)=-1$, then $u_1$, \ldots, $u_{t-1}$ are $4$-vertices.
\end{proposition}
\begin{proof}
Let $\beta_0=\beta_t=0$ and $\beta_i=1$ for $1\le i \le t-1$.
Note that $\ch(S)=-1+\sum_{i=0}^t (\ch(u_i\to f)-\beta_i)$.
For $1\le i\le t-1$, the edges $u_{i-1}u_i$ and $u_iu_{i+1}$ are incident with $3$-faces, and thus $u_i$ is a $4^+$-vertex type-II incident with $f$;
hence, we have $\ch(u_i\to f)\ge 1$ by R4(II), R5(II), and R6(II).
Consequently, $\ch(u_i\to f)-\beta_i\ge 0$ for $0\le i\le t$ and $\ch(S)\ge -1$.
Since $S$ is negative, we have $\ch(u_i\to f)<\beta_i+1$ for $0\le i\le t$,
and by R6(II), R6(I), R5(I), and R4(I), we conclude that all vertices of $S$ are internal of degree at most $5$,
both $u_0$ and $u_t$ have degree at most $4$, and if they are $4$-vertices, then they are type-I-1 or type-I-2 incident with $f$.
Furthermore, by R5(II), if $u_i$ is a $5$-vertex for some $i\in\{1,\ldots, t-1\}$, then $u_i$ is type-I-3 incident with another
$6^+$-face, so that $\ch(u_i\to f)=\beta_i$.
By R4(I), if $u_i$ is a $4$-vertex for some $i\in\{0,t\}$, then $\ch(u_i\to f)=\beta_i+1/2$, and thus
either both $u_0$ and $u_t$ are $3$-vertices and $\ch(S)=-1$, or $u_0$ is a $3$-vertex and $u_t$ is a $4$-vertex
and $\ch(S)=-1/2$.  In the latter case, $u_t$ is type-I-1 or type-I-2 incident with $f$, and in particular $u_{t-1}$ is
a $4^+$-vertex, and consequently $t\ge 2$.

Suppose now that some vertex $u_i$ with $i\in\{1,\ldots,t-1\}$ is a $5$-vertex; as we observed,
$u_i$ is type-I-3 incident with another $6^+$-face.  By Lemma~\ref{5-vertex3}, we have $i\ge 2$,
and by Lemma~\ref{5-vertex43}, we have $i\ge 3$.
If $\ch(S)=-1$, then $u_t$ is a $3$-vertex, and a symmetric
argument shows that neither $u_{t-1}$ nor $u_{t-2}$ is a $5$-vertex.
Consequently, if either $t\le 3$, or $t\le 5$ and $\ch(S)=-1$, then $u_1$, \ldots, $u_{t-1}$ are $4$-vertices.
\end{proof}

We say two segments of the same $6^+$-face $f$ are \emph{adjacent} if an edge of the cycle bounding $f$ joins their ends.

\begin{proposition}
\label{charge}
Let $(G,L,Z)$ be a minimal counterexample and let $f$ be a $6^+$-face of $G$. The following propositions hold.
\begin{itemize}
\item[(1)] If $S$ is a segment of $f$ adjacent to a negative $1$-segment, then $\ch(S)\ge 0$. Additionally,
if $S$ is a $1$-segment, then $\ch(S)\ge 1/2$.
\item[(2)] Suppose $uvw$ is a subpath of the cycle bounding $f$, where $uv$ is incident with a $3$-face.
If $w$ is type-0 incident with $f$ and $v$ is a $4^+$-vertex, then
$\ch(v\to f)+\ch(w\to f)\geq 1$ (see Figure~\ref{fig-charge}(a)).
\item[(3)] Suppose $uvw$ is a subpath of the cycle bounding $f$, where both $v$ and $w$ are type-0 incident with $f$
and $v$ is a $4$-vertex.
Let $u$, $w$, $x$, $y$ be the neighbors of $v$ listed in cyclic order according to their drawing in $G$.
If $u$ is a $3$-vertex, $w$ is a $5^+$-vertex, and both $x$ and $y$ are $4^+$-vertices,
then $\ch(v\to f)+\ch(w\to f)\geq 1$ (see Figure~\ref{fig-charge}(b)).
\item[(4)] Suppose $uvwx$ is a subpath of the cycle bounding $f$, where $uv$ is incident with a $3$-face
and $u$ and $v$ are $3$-vertices.  If $wx$ is incident with a $3$-face, then let $T=\{w\}$, otherwise let $T=\{w,x\}$.
Then either $\sum_{z\in T} \ch(z\to f)\ge 1$ or both $w$ and $x$ are $4$-vertices type-0 incident with $f$.
\end{itemize}

\end{proposition}
\begin{proof}
Let us prove the claims separately.
\begin{itemize}
\item[(1)] Let $S=u_0\ldots u_t$, and let $S'=v_0v_1$ be a negative $1$-segment adjacent to
$S$; say $v_1u_0$ is an edge of the cycle bounding $f$.  By Proposition~\ref{segment},
both $v_0$ and $v_1$ are $3$-vertices.  Note that $u_0$ is not adjacent to $v_0$,
since all triangles in $G$ bound faces and $\deg(v_1)>2$.  By Lemma~\ref{33-path},
$u_0$ is a $4^+$-vertex.  Since $v_1$ is a $3$-vertex, $u_0$ is neither type-I-1 nor type-I-2
incident with $f$, and thus $\ch(S)\ge 0$ by Proposition~\ref{segment}.

Suppose now that $t=1$.  If $\ch(u_0\to f)\ge 3/2$, then $\ch(S)\ge \ch(u_0\to f)-1=1/2$.
Hence, we can assume that $\ch(u_0\to f)<3/2$, and thus by R6(I), $u_0$ is not a $6^+$-vertex.
Since $u_0$ is neither type-I-1 nor type-I-2 incident with $f$, R4(I) and R5(I) imply
$\ch(u_0\to f)=1$.
If $u_0$ is a $5$-vertex, this by R5(I) implies that $u_0$ is not type-I-3 incident with $f$, and
thus $u_1$ is a $4^+$-vertex.  If $u_0$ is a $4$-vertex, then by Lemma~\ref{33-path},
we again conclude that $u_1$ is a $4^+$-vertex.  In either case, R4(I), R5(I), and R6(I)
imply $\ch(u_1\to f)\ge 1/2$, and thus $\ch(S)=\ch(u_0\to f)+\ch(u_1\to f)-1\ge 1/2$.

\item[(2)] By R4(I), R5(I), and R6(I), we have $\ch(v\to f)\ge 1$, unless $v$ is a $4$-vertex type-I-1 or type-I-2
incident with $f$.  If $v$ is type-I-1 or type-I-2 incident with $f$, then $\ch(v\to f)=1/2$ and $w$ is a $5^+$-vertex.
Note that in this case $w$ is not type-0-3 incident with $f$ (this is clear if $v$ is type-I-2 incident with $f$,
and follows by Lemma~\ref{5-vertex43} if $v$ is type-I-1 incident with $f$).
Hence, $\ch(w\to f)\ge 1/2$ by R5(0) and R6(0), and $\ch(v\to f)+\ch(w\to f)\geq 1$.

\item[(3)] If $w$ is a $6^+$-vertex, then $\ch(w\to f)=1$ by R6(0).  Hence, assume that $w$ is a $5$-vertex.
By Lemma~\ref{5-vertex43}, $w$ is not type-0-3 incident with $f$, and thus $\ch(w\to f)\ge 1/2$ by R5(0).
If $xy\in E(G)$, then consider the face $g$ whose boundary contains the path $wvx$, and observe that
$v$ is type-I-1 incident with $g$.  If $xy\not\in E(G)$, then $v$ is not incident with any $3$-faces.
In either case, $\ch(v\to f)=1/2$ by R4(0), and thus $\ch(v\to f)+\ch(w\to f)\ge 1$.

\item[(4)] By Lemma~\ref{33-path}, $w$ is a $4^+$-vertex.  If $w$ is a $5$-vertex, then it is either type-I or type-0-2
incident with $f$.  Hence, if $w$ is a $5^+$-vertex, then $\ch(w\to f)\ge 1$ by R5(I), R5(0), and R6.  Consequently,
we can assume that $w$ is a $4$-vertex.  If $wx$ is incident with a $3$-face, then note that $w$ is neither
type-I-1 nor type-I-2 incident with $f$, and $\ch(w\to f)=1$ by R4(I).  Hence, assume that $wx$ is not incident with a $3$-face.
By Lemma~\ref{33-path}, $x$ is a $4^+$-vertex.
If $x$ is a $5^+$-vertex, then note that $w$ has no $3$-vertex neighbors other than $v$ by Lemma~\ref{33-path},
and thus $\ch(w\to f)+\ch(x\to f)\geq 1$ by (3).
If $x$ is a $4$-vertex type-I incident with $f$, then note that $x$ is neither type-I-1 nor type-I-2 incident with $f$,
and thus $\ch(x\to f)=1$ by R4(I).  Therefore, either $\sum_{z\in T} \ch(z\to f)\ge 1$ or both $w$ and $x$ are $4$-vertices
type-0 incident with $f$.
\end{itemize}
\end{proof}

\begin{figure}[!htb]
\centering
{\includegraphics[height=0.35\textwidth]{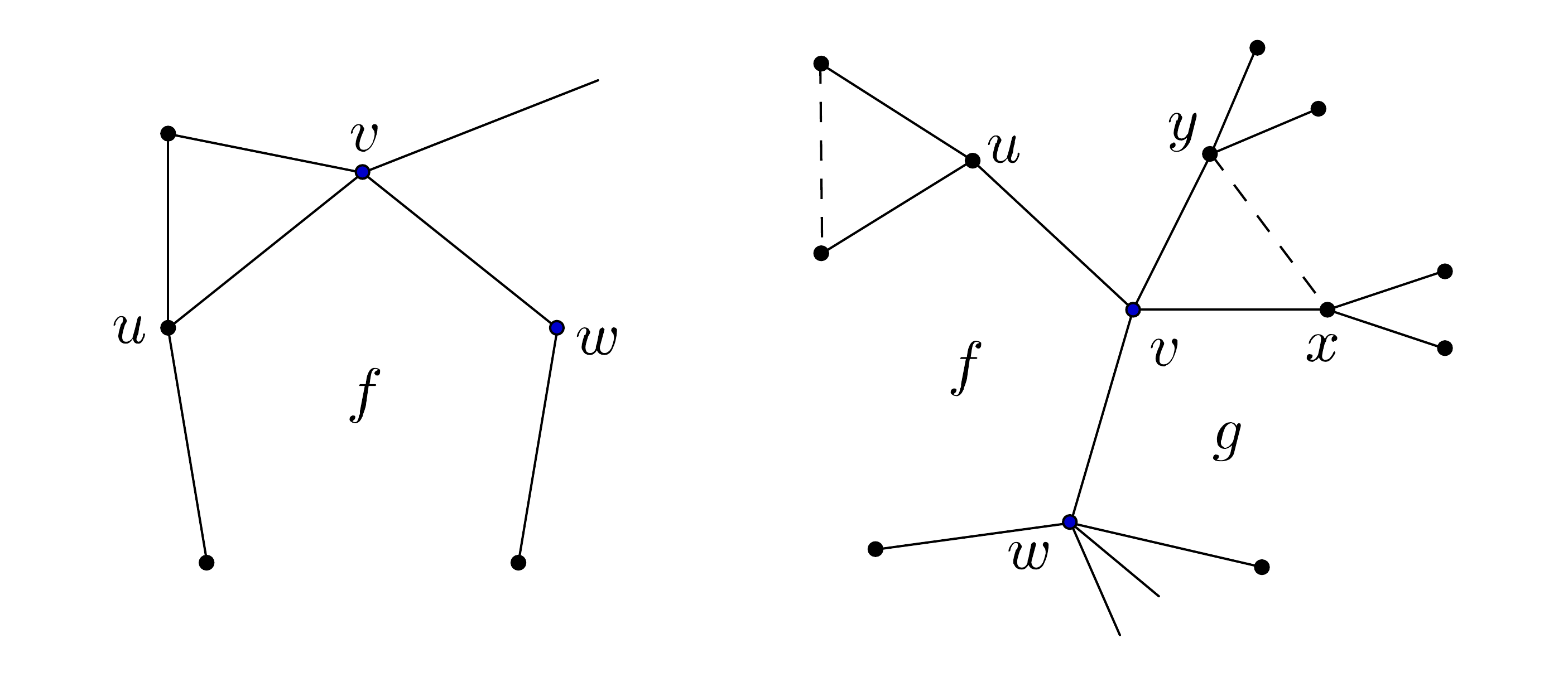}}

(a)  \hspace{3.0cm}  (b) \hspace{1.0cm}
\caption{Configurations from Proposition~\ref{charge}.}\label{fig-charge}
\end{figure}

\begin{lemma}\label{6face}
Let $(G,L,Z)$ be a minimal counterexample and let $f$ be a face of $G$.
If $f$ has length $6$, then $\ch(f)\ge 0$.
\end{lemma}
\begin{proof}
Note that $\ch_0(f)=0$.
Let $v_1\ldots v_6$ be the cycle bounding $f$.  If all faces that share edges with $f$ are $3$-faces,
then $v_1$, \ldots, $v_6$ are $4^+$-vertices type-II incident with $f$, and thus
$\ch(f)\ge 6\times 1-6\times 1=0$ by R4(II), R5(II), R6(II), and Rt.
Hence, we can assume that $f$ shares an edge with a $6^+$-face.  Hence, each edge which $f$ shares with
a $3$-face is contained in a segment.  Let $S_1$, \ldots, $S_k$ be segments of $f$.
By Rt, we have $\ch(f)\ge \sum_{i=1}^k \ch(S_i)$, and thus we can assume that say $S_1$ is a negative segment.
We can label the vertices so that $S_1=v_1\ldots v_m$ for some $m\ge 2$.

Let us first consider the case that $\ch(S)\ge -1/2$ for every segment $S$ of $f$.
By Proposition~\ref{segment}, we have $\ch(S_1)=-1/2$, $m\ge 3$, and we can assume that $v_1$
is a $3$-vertex and $v_m$ is a $4$-vertex type-I-1 or type-I-2 incident with $f$.  Consequently,
$v_{m+1}$ is a $5^+$-vertex (and in particular $m\le 5$).  If $v_{m+1}$ is type-0 incident
with $f$, then $f$ cannot have a negative segment other than $S_1$ (as it would be a $1$-segment,
which cannot have charge at least $-1/2$ by Proposition~\ref{segment}).  Furthermore,
$\ch(v_m\to f)+\ch(v_{m+1}\to f)\ge 1$ by Proposition~\ref{charge}(2), and thus
$\ch(f)\ge (\ch(S_1)-\ch(v_m\to f))+(\ch(v_m\to f)+\ch(v_{m+1} \to f))\ge -1+1=0$.
Hence, we can assume that $v_{m+1}$ is type-I incident with $f$ and starts a segment $S_2=v_{m+1}\ldots v_i$ for some $i\in \{5,6\}$.
We have $\ch(v_{m+1}\to f)\ge 1$ by R5(I) and R6(I), and if $v_i$ is a $4^+$-vertex, then $\ch(v_i\to f)\ge 1/2$ by R4(I), R5(I), and R6(I), and thus
$\ch(S_2)\ge 1/2$, and $\ch(f)\ge \ch(S_1)+\ch(S_2)\ge 0$.  Hence, suppose that $v_i$ is a $3$-vertex.
Note that $m\le 4$, and thus $v_2$, \ldots, $v_{m-1}$ are $4$-vertices by Proposition~\ref{segment}.
If $m=3$, then by Lemmas~\ref{33-path} and \ref{34-path}, we conclude that
$i=5$ and $v_6$ is a $5^+$-vertex, not type-0-3 incident with $f$ by Lemma~\ref{5-vertex3}.
Hence, $\ch(f)=\ch(S_1)+\ch(S_2)+\ch(v_6\to f)\ge -1/2+0+1/2=0$ by R5(0) and R6(0).
Therefore, we can assume that $m=4$, and thus $i=6$.  By Lemma~\ref{33-path},
$v_4$ is not type-I-1 incident with $f$, and thus it is type-I-2 incident; i.e.,
$f$ shares the edge $v_4v_5$ with a $6$-face bounded by a cycle $v_4v_5w_1yx_1x$, where the edge $w_1y$ is incident
with a $3$-face bounded by cycle $w_1yy_2$, $x_1$ and $w_1$ are $3$-vertices and $y$ is a $4$-vertex.  Note that
$y_2$ is a $4^+$-vertex by Lemma~\ref{4-vertex}.  Consequently, $v_5$ is either a $6^+$-vertex, or a $5$-vertex
type-I-4 incident with $f$, and $\ch(v_5\to f)=3/2$ by R5(I) and R6(I).  Consequently, $\ch(S_2)=1/2$
and $\ch(f)=\ch(S_1)+\ch(S_2)=0$.

Let us now consider the case that $f$ is incident with a segment $S$ with $\ch(S)<-1/2$, say $S=S_1$.
By Proposition~\ref{segment}, $\ch(S_1)=-1$ and both $v_1$ and $v_m$ are $3$-vertices.
Since $m\le 6$, Proposition~\ref{segment} also implies that $v_2$, \ldots, $v_{m-1}$ are $4$-vertices.
By Lemma~\ref{6-cycle}, we have $m\le 5$.  If $m=5$, then by Lemmas~\ref{33-path} and
\ref{6-cycle}, $v_6$ is a $5^+$-vertex.  Note that if $v_6$ is a $5$-vertex, then it is
type-0-1 incident with $f$.  Hence, $\ch(f)=\ch(S_1)+\ch(v_6\to f)=-1+1=0$ by R5(0) and R6(0).
Let us distinguish cases depending on whether $m$ is $2$, $3$, or $4$.

\smallskip

\textbf{Case $m=4$:} By Lemmas~\ref{33-path} and \ref{6-cycle}, we can assume that $v_5$ is a $4^+$-vertex
and $v_6$ is a $5^+$-vertex.  If $v_5v_6$ is incident with a $3$-face (and thus $v_5v_6$ is a segment $S_2$),
then note that $v_5$ is not type-I-1 or type-I-2 incident with $f$, and thus $\ch(v_i\to f)\ge 1$
for $i\in\{5,6\}$ by R4(I), R5(I), and R6(I), $\ch(S_2)\ge 1$, and $\ch(f)=\ch(S_1)+\ch(S_2)\ge -1+1=0$.
Hence, suppose that $v_5$ and $v_6$ are type-0 incident with $f$.  Note that neither $v_5$ nor $v_6$ is
type-0-3 incident with $f$ by Lemma~\ref{5-vertex3}, and thus $\ch(v_5\to f)+\ch(v_6\to f)\ge 1$ by R5(0) and R6(0)
unless $v_5$ is a $4$-vertex.  If $v_5$ is a $4$-vertex, then note that $v_4$ is the only $3$-vertex neighbor of
$v_5$ by Lemma~\ref{34-path}, and thus $\ch(v_5\to f)+\ch(v_6\to f)\ge 1$ by Proposition~\ref{charge}(3).
Hence, $\ch(f)=\ch(S_1)+\ch(v_5\to f)+\ch(v_6\to f)\ge -1+1=0$.

\smallskip

\textbf{Case $m=3$:}  By Lemma~\ref{33-path}, $v_4$ and $v_6$ are $4^+$-vertices.
If $v_5$ is a $3$-vertex, then $v_4$ and $v_6$ are $5^+$-vertices
by Lemma~\ref{33-path} and by symmetry we can assume that $v_4$ is type-0 incident with $f$.  Then
$S_1$ is the only negative segment of $f$ ($v_5v_6$ could be a $1$-segment, but by Proposition~\ref{charge}(1),
a negative $1$-segment cannot be adjacent to $S_1$) and if $v_4$ is a $5$-vertex, then it is type-0-1
incident with $f$.  Hence, $\ch(f)\ge \ch(S_1)+\ch(v_4\to f)\ge -1+1=0$ by R5(0) and R6(0).  Consequently, we can assume
that $v_5$ is also a $4^+$-vertex.

Suppose that $f$ is incident with a segment $S_2\neq S_1$.  If $\ch(S_2)\ge 1$, then $\ch(f)\ge \ch(S_1)+\ch(S_2)\ge -1+1=0$.
Hence, we can assume that $\ch(S_2)<1$.  Note that neither $v_4$ nor $v_6$ is type-I-1 or type-I-2 incident with $f$,
and thus by R4(I), R5(I), and R6(I), this is only possible if $v_5$ is an end of $S_2$ (say $S_2=v_5v_6)$, $v_5$ is a $4$-vertex
and $v_5$ is type-I-1 or type-I-2 incident with $f$.  Then $\ch(S_2)=1/2$ and $v_4$ is a $5^+$-vertex.  By Lemma~\ref{5-vertex3},
$v_4$ is not type-0-3 incident with $f$, and thus $\ch(v_4\to f)\ge 1/2$ by R5(0) and R6(0).  Hence,
$\ch(f)=\ch(S_1)+\ch(S_2)+\ch(v_4\to f)\ge -1+1/2+1/2=0$.

Hence, we can assume that $S_1$ is the only segment of $f$.  If $\ch(v_4\to f)\ge 1/2$ and $\ch(v_6\to f)\ge 1/2$,
then $\ch(f)\ge \ch(S_1)+\ch(v_4\to f)+\ch(v_6\to f)\ge -1+2\times 1/2=0$.  Hence, by symmetry we can assume that $\ch(v_4\to f)<1/2$.
By Lemma~\ref{5-vertex3}, $v_4$ is not type-0-3 incident with $f$, and thus by R4(0), R5(0), and R6(0), we conclude that
$v_4$ is a $4$-vertex.  By Lemma~\ref{34-path}, $v_3$ is the only $3$-vertex neighbor of $v_4$.  If
$v_5$ is a $5^+$-vertex, then $\ch(v_4\to f)+\ch(v_5\to f)\ge 1$ by Proposition~\ref{charge}(3), and
thus $\ch(f)\ge \ch(S_1)+\ch(v_4\to f)+\ch(v_5\to f)\ge -1+1=0$.  Hence, we can assume that $v_5$ is a $4$-vertex,
and thus $v_6$ is a $5^+$-vertex by Lemma~\ref{6-cycle}.  By Lemma~\ref{5-vertex3}, $v_6$ is not type-0-3 incident with $f$,
and thus $\ch(v_6\to f)\ge 1/2$ by R5(0) and R6(0).
Let $f'$ denote the face with that $f$ shares the edge $v_5v_6$, and let $z\neq v_6$ be the neighbor of $v_5$ in the boundary
cycle of $f'$.  By Lemma~\ref{34-path}, $z$ is a $4^+$-vertex, and thus if $v_5$ is incident with a $3$-face, then
$v_5$ is type-I-2 incident with $f'$.  Consequently, $\ch(v_5\to f)=1/2$ by R4(0),
and $\ch(f)\ge \ch(S_1)+\ch(v_5\to f)+\ch(v_6\to f)\ge -1+2\times 1/2=0$.

\smallskip

\textbf{Case $m=2$:}
If $k=3$, then the segments of $f$ are $S_1$, $S_2=v_3v_4$ and $S_3=v_5v_6$.
By Proposition~\ref{charge}(1), $\ch(S_2),\ch(S_3)\ge 1/2$, and thus
$\ch(f)=\ch(S_1)+\ch(S_2)+\ch(S_3)=0$.  Hence, we can assume that $k\le 2$.
If $v_3$ is contained in a segment, then let $T_3=\{v_3\}$, otherwise let $T_3=\{v_3,v_4\}$.
If $v_6$ is contained in a segment, then let $T_6=\{v_6\}$, otherwise let $T_6=\{v_5,v_6\}$.
For $i\in\{3,6\}$, let $\gamma_i=\sum_{x\in T_i} \ch(x\to f)$.
By Lemma~\ref{6-cycle}, $v_3$, \ldots, $v_6$ cannot all be $4$-vertices, and thus
by Proposition~\ref{charge}(4), we have $\max(\gamma_3,\gamma_6)\ge 1$.
If $k=1$, then $\ch(f)=\ch(S_1)+\gamma_3+\gamma_6\ge -1+0+1=0$.
Hence, we can assume that $k=2$.  Let $S_2\neq S_1$ be the other segment of $f$, with ends $x$ and $y$,
and let $\beta=\ch(S_2)-\ch(x\to f)-\ch(y \to f)$.  Observe that $\beta\ge -1$, and
$\ch(f)\ge \ch(S_1)+\beta+\gamma_3+\gamma_6$.  If $\gamma_3,\gamma_6\ge 1$, we have $\ch(f)\ge 0$.
Hence, by symmetry we can assume that $\gamma_3<1$, and by Proposition~\ref{charge}(4),
$v_3$ and $v_4$ are $4$-vertices type-0 incident with $f$, and thus $S_2=v_5v_6$.
By Lemma~\ref{33-path}, $v_5$ and $v_6$ are $4^+$-vertices, and clearly neither of them is type-I-1 or type-I-2
incident with $f$.  By R4(I), R5(I), and R6(I), we have $\ch(v_i\to f)\ge 1$ for $i\in \{5,6\}$,
and thus $\ch(S_2)\ge 1$.  Consequently, $\ch(f)\ge \ch(S_1)+\ch(S_2)\ge 0$.
\end{proof}

\begin{lemma}\label{7face}
Let $(G,L,Z)$ be a minimal counterexample and let $f$ be a face of $G$.
If $f$ has length at least $7$, then $\ch(f)\ge 0$.
\end{lemma}
\begin{proof}
Let $C=v_1\ldots v_m$ be the cycle bounding $f$.  If all faces that share edges with $f$ are $3$-faces,
then $v_1$, \ldots, $v_m$ are $4^+$-vertices type-II incident with $f$, and thus
$\ch(f)\ge \ch_0(f)+m\times 1-m\times 1\ge 0$ by R4(II), R5(II), R6(II), and Rt.
Hence, we can assume that $f$ shares an edge with a $6^+$-face.  Hence, each edge which $f$ shares with
a $3$-face is contained in a segment.  Let $S_1$, \ldots, $S_k$ be the segments of $f$, and let $n$ denote the number
of them which are negative.

For a negative segment $S=v_iv_{i+1}\ldots v_s$,
we say that $S$ \emph{owns} the edges of the path $S$, the edge $v_sv_{s+1}$, and if $S$ is a $1$-segment,
then also the edge $v_{i-1}v_i$ (with all indices taken cyclically modulo $m$).  Note that by Proposition~\ref{segment}
and Lemma~\ref{33-path}, each edge of $C$ is owned by at most one negative segment, and each negative segment
owns at least three edges.  Consequently, $n\le \lfloor m/3\rfloor$,
and
$$\ch(f)\ge \ch(f_0)+\sum_{i=1}^k \ch(S_i)\ge \ch(f_0)-n\ge m-6-\lfloor m/3\rfloor.$$
It follows that if $m\ge 8$, then $\ch(f)\ge 0$.

Hence, suppose that $|f|=7$, and thus $n\le 2$.  If $f$ has at most one negative segment, or two negative segments of
charge at least $-1/2$, then $\ch(f)\ge \ch(f_0)-1=0$.  Hence, we can assume that $f$ has two negative segments $S_1$ and $S_2$
and $\ch(S_1)<-1/2$.  By Proposition~\ref{segment}, $\ch(S_1)=-1$, and we can assume $S_1=v_1\ldots v_s$ for some $s\le 5$,
$v_1$ and $v_s$ are $3$-vertices, and $v_2$, \ldots, $v_{s-1}$ are $4$-vertices.  By Lemma~\ref{33-path}, $v_{s+1}$ and $v_7$
are $4^+$-vertices; furthermore, they are clearly neither type-I-1 nor type-I-2 incident with $f$.
Since $S_2$ is negative, Proposition~\ref{segment} implies that $v_{s+1},v_7\not\in V(S_2)$, and thus $s\le 3$.
If $\ch(S_2)>-1$, then by Proposition~\ref{segment} we conclude that $s=2$ and $S_2$ is the $2$-segment $v_4v_5v_6$,
and one end of $S_2$ is a $3$-vertex; otherwise, Proposition~\ref{segment} implies that both ends of $S_2$ are $3$-vertices.
By symmetry, we can assume that $v_6$ is a $3$-vertex belonging to $S_2$.  By Lemma~\ref{34-path}, $v_7$ is a $5^+$-vertex.
Note that if $v_7$ is a $5$-vertex, then it is type-0-1 incident with $f$, and thus $\ch(v_7\to f)=1$ by R5(0) and R6(0).
Hence, $\ch(f)\ge \ch_0(f)+\ch(S_1)+\ch(S_2)+\ch(v_7\to f)\ge 1-2\times 1+1=0$.
\end{proof}

\begin{lemma}
\label{face}
Let $(G,L,Z)$ be a minimal counterexample.  Then every face $f$ of $G$ satisfies $\ch(f)\ge 0$.
\end{lemma}
\begin{proof}
If $f$ is a $6^+$-face, this follows from Lemmas~\ref{6face} and \ref{7face}.
If $f$ is a $3$-face, then note that $f$ only shares edges with $6^+$-faces by the
absence of $4$- and 5-cycles, and thus $\ch(f)=\ch_0(f)+3\times 1=0$ by Rt.
\end{proof}

\section{$(11:3)$-colorability of planar graphs}

We are now ready to prove our main result.

\begin{proof}[Proof of Theorem~\ref{MTl}]
Suppose for a contradiction that there exists a plane graph $G_0$ without $4$- or $5$-cycles
and an assignment $L_0$ of lists of size 11 to vertices of $G_0$ such that
$G_0$ is not $(L_0:3)$-colorable.  Let $z$ be any vertex of $G_0$, let $L_0'(z)$ be any $3$-element
subset of $L_0(z)$, and let $L'_0(v)=L_0(v)$ for all $v\in V(G)\setminus\{z\}$.
Then $G_0$ is not $(L'_0:3)$-colorable, and thus $(G_0, L_0,\{z\})$ is a counterexample.

Therefore, there exists a minimal counterexample $(G,L,Z)$.
Let $\ch$ be the assignment of charges to vertices and faces of $G$ obtained
from the initial charge $\ch_0$ as described in Section~\ref{sec-discharge}.
By (\ref{eq:sum}), the fact that the total amount of charge does not change
by its redistribution, and Lemmas~\ref{vertex} and \ref{face}, we have
$$-12=\sum_{v\in V(G)}\ch_0(v)+\sum_{f\in F(G)} \ch_0(f)=\sum_{v\in V(G)}\ch(v)+\sum_{f\in F(G)} \ch(f)\ge \sum_{z\in Z} (\deg(z)-6).$$
Since $|Z|\le 3$ and $\deg(z)\ge 2$ for all $z\in Z$ by Lemma~\ref{conn},
we conclude that $|Z|=3$ and all vertices of $Z$ have degree two.  But since $G$
is connected and $G[Z]$ is a triangle, this implies that $V(G)=Z$, and thus $G$ is $(L:3)$-colorable.
This is a contradiction.
\end{proof}

\section*{Acknowledgments}

The research leading to these results has received funding from the European
Research Council under the European Union's Seventh Framework Programme
(FP/2007-2013)/ERC Grant Agreement n.616787. Xiaolan Hu is partially supported
by NSFC under grant number 11601176 and NSF of Hubei Province under grant
number 2016CFB146.

\bibliographystyle{siam}
\bibliography{../data}

\end{document}